\documentclass[12pt]{amsart}
\usepackage[margin=1.1in]{geometry}
\usepackage{amssymb,amsthm,hyperref}
\usepackage{epsfig}
\usepackage[arrow]{xy}

\newtheorem{lem}{Lemma}
\newtheorem{lemma}[lem]{Lemma}
\newtheorem{prop}[lem]{Proposition}
\newtheorem{proposition}[lem]{Proposition}
\newtheorem{theorem}[lem]{Theorem}
\newtheorem{conjecture}[lem]{Conjecture}
\newtheorem{corollary}[lem]{Corollary}

\newtheorem{problem}[lem]{Problem}

\theoremstyle{remark}
\newtheorem{remark}[lem]{Remark}
\newtheorem{example}[lem]{Example}

\numberwithin{lem}{section}

\numberwithin{equation}{section}

\def\A{{\mathcal A}}
\def\Z{{\mathbb Z}}
\def\L{{\mathcal L}}

\def\P{{\mathcal P}}
\def\x{{\bf x}}
\def\F{{\bf F}}
\def\FF{{\mathcal F}}
\def\hF{{\hat F}}
\def\hP{{\hat P}}
\def\hQ{{\hat Q}}
\def\hR{{\hat R}}
\def\hS{{\hat S}}
\def\hT{{\hat T}}
\def\Frac{{\rm Frac}}

\def\Y{{\mathfrak Y}}
\def\y{{\mathfrak y}}
\def\X{{\tilde X}}
\def\YY{{\mathcal C}}
\def\Yy{{\mathcal P}}
\def\y{{\mathfrak y}}
\def\I{{\mathcal I}}
\def\s{{\mathcal S}}
\def\C{{\mathfrak C}}
\def\den{\mathrm{den}}
\def\T{{\mathcal T}}
\def\N{{\mathcal N}}
\def\hN{\tilde \N}

\title{Linear Laurent phenomenon algebras}

\author{Thomas Lam}\address
 {Department of Mathematics\\ University of Michigan\\ Ann Arbor\\ MI 48109 USA.}
 \date{\today}
 \email{tfylam@umich.edu}
 \urladdr{http://www.math.lsa.umich.edu/\~{ }tfylam}
 \thanks{T.L. was supported by NSF grant DMS-0901111 and by a Sloan Fellowship.}
  \author{Pavlo Pylyavskyy}\address
{Department of Mathematics\\ University of Minnesota\\ Minneapolis\\ MN 55414 USA.}
 \email{ppylyavs@umn.edu}
 \urladdr{http://sites.google.com/site/pylyavskyy/}
 \thanks{P.P. was supported by NSF grant DMS-0757165.}

\begin{document}
\begin{abstract}
In \cite{LP} we introduced Laurent phenomenon algebras, a generalization of cluster algebras.  Here we give an explicit description of Laurent phenomenon algebras with a linear initial seed arising from a graph.  In particular, any graph associahedron is shown to be the dual cluster complex for some Laurent phenomenon algebra.
\end{abstract}
\maketitle

\section{Introduction}
\label{sec:intro}
Cluster algebras were introduced by Fomin and Zelevinsky in \cite{CA1}. They were quickly recognized to be rather ubiquitous throughout mathematics, appearing for example in 
representation theory of quivers and finite-dimensional algebras, Poisson geometry, Teichm\"uller theory, integrable systems,  and the study of Donaldson-Thomas invariants.

In \cite{LP}, we introduced  {\it Laurent phenomenon algebras} (LP algebras), a generalization of cluster algebras where exchange polynomials  were allowed to have arbitrarily 
many monomials, rather than being just binomials.  The aim of this article is to describe certain LP algebras with a seed where the exchange polynomials are linear.  These LP algebras give a large family of examples of finite type LP algebras  of arbitrarily large rank.

%\fixit{General motivation for LP and cluster algebras?}

One of the highlights of Fomin and Zelevinsky's theory of cluster algebras is the classification \cite{FZ1} of finite type cluster algebras, which turns out to be identical to the Cartan-Killing classification of semisimple Lie algebras.  The combinatorial structure of finite type cluster algebras is captured by remarkable polytopal-complexes known as {\it generalized associahedra} \cite{CFZ,FZ2}.

The cluster complexes of the LP algebras we study here are not known to be polytopal, but they contain subcomplexes which are polytopal and studied previously: the {\it graph associahedra} and {\it nestohedra} studied in \cite{CD,FS,Pos, Zel}.  In particular, any nestohedron arising from a directed graph is the exchange polytope of some LP algebra.
This explains and confirms the ``striking similarity'' Zelevinsky has observed between nested complexes and cluster complexes \cite{Zel}.

We now describe the results of this paper in more detail.

\subsection{Laurent phenomenon algebras}
Let $R$ be a coefficient ring containing algebraically independent elements $A_1,A_2,\ldots,A_n$.  For example, $R$ could be $\Z[A_1,\ldots,A_n]$.  Let $\FF = \Frac(R[X_1,X_2,\ldots,X_n])$ be the ambient field, where $X_1,\ldots,X_n$ are indeterminates.

A Laurent phenomenon algebra is a subring $\A \subset \FF$ together with a collection of seeds $t = (\x,\F)$, where $\x=\{x_1,\ldots,x_n\} \subset \FF$ is a set of variables, and $\F=\{F_1,\ldots,F_n\} \subset R[x_1,\ldots,x_n]$ is a set of exchange polynomials, satisfying conditions that shall be recalled in Section \ref{sec:background}.  The collection of seeds are connected by mutation.  For each $i \in [n] = \{1,2,\ldots,n\}$, and a seed $t = (\x,\F)$, we have a mutated seed $t' = \mu_i(t) = (\x',\F')$.  The cluster variables of $t$ and $t'$ are the same except for $x_i$ and $x'_i$, which satisfy an exchange relation $x_ix'_i = \hF_i$, where $\hF_i$ is the exchange Laurent polynomial associated to $F_i$.  For the reader comfortable with cluster algebras, the key difference is that $\hF_i$ is a Laurent polynomial rather than a non-Laurent binomial.

The {\it cluster complex} of a LP algebra is the simplicial complex with base set equal to the set of cluster variables, and faces corresponding to collections of cluster variables that lie in the same cluster.  The {\it exchange graph} of a LP algebra $\A$ is the graph with vertex set equal to the set of seeds of $\A$, and edges given by mutations.

\subsection{Nested collections}
\label{ss:nested}
Let $\Gamma$ be a directed unweighted graph on $[n]$.  A non-empty subset $I \subset [n]$ is {\it strongly connected} if the induced subgraph of $\Gamma$ on $I$ is strongly connected; that is, there is a directed path between every pair of vertices in $I$.  We let $\I \subset 2^{[n]}$ denote the collection of strongly-connected subsets. A family of subsets $\s = \{I_1, \ldots, I_k\} \in \I$ is {\it nested} if 
\begin{itemize}
 \item for any pair $I_i, I_j$ either one of them lies inside the other, or they are disjoint;
 \item for any tuple of disjoint $I_j$-s, they are the strongly connected components of their union.
\end{itemize}
See Example \ref{ex:nested} for an example.  The {\it support} $S \subset [n]$ of a nested family $\s$ is given by $S:= \bigcup_k I_k$.  Denote by $\T \subset \I$ the set of strongly connected components of $\Gamma$.  If $\Gamma$ is strongly connected, then $\T = \{[n]\}$.

The nested families with maximal support ($S = [n]$) form a simplicial complex $\N=\N(\Gamma)$, called the {\it nested set complex} by Feichtner and Sturmfels \cite{FS} and the {\it nested complex} by Postnikov \cite{Pos}.  
The base set of $\N$ is the set of $\I \setminus \T$, and a subset of $\I \setminus \T$ is a face if it is nested.  These complexes were studied in a more general context in \cite{FS, Pos}.  To each such simplicial complex $\N(\Gamma)$, there is an associated polytope $P(\Gamma)$, called  a {\it nestohedron}, whose face lattice is dual to the nested complex.  When $\Gamma$ is an undirected graph, the nestohedron is known as a {\it graph associahedron}, studied by Carr and Devadoss \cite{CD}.  This family of polytopes includes the well-known associahedron and cyclohedron.

The {\it extended nested complex} $\hN=\hN(\Gamma)$ is the simplicial complex on $\I$ consisting of all nested families.

\subsection{Acyclic functions}\label{ss:Y}
We define certain generating functions of acyclic subgraphs in $\Gamma$.  For each non-empty subset $I \in 2^{[n]}$ let us consider functions $f\colon I \longrightarrow [n]$ such that $i \to f(i)$ is an edge of $\Gamma$ for each $i$, with the additional possibility that $f(i) = i$.  Call such a function {\it {acyclic}} if the only directed cycles in the resulting graph are loops $i \mapsto f(i)$.

For a strongly-connected component $I \in 2^{[n]}$ define the following Laurent polynomials in $X_1,X_2,\ldots,X_n$:
$$
Y_I = \frac{\sum_{\text{acyclic } f\colon I \longrightarrow [n]} \prod_{i \in I} \X_{f(i)}}{\prod_{i \in I} X_i}
$$ 
where
$$
\X_{f(i)} = 
\begin{cases}
X_{f(i)} & \text{if $i \not = f(i)$}\\
A_{f(i)} & \text{if $i = f(i)$.}
\end{cases}
$$
An example is given in Example \ref{ex:Y}, and further details on the $Y_I$ rational functions are given in Section \ref{sec:Y}.

\subsection{Graph LP algebras}
Given a directed graph $\Gamma$, we define an initial seed $t_\Gamma$ (also denoted $t_\emptyset$ in Section \ref{sec:main}) with cluster variables $\{X_1,\ldots,X_n\}$ and exchange polynomials $F_i = A_i + \sum_{i \to j} X_j$, where $i \to j$ denotes an edge in $\Gamma$.  Let $\A_\Gamma = \A(t_\Gamma)$ denote the LP algebra generated by the initial seed $t_\Gamma$.  Our main theorem is

\begin{theorem}\label{thm:AGamma}
The normalized LP algebra $\A_\Gamma = \A(t_\Gamma)$ 
\begin{itemize}
\item has cluster variables 
$$
\{X_1,X_2,\ldots,X_n\} \cup \{Y_I \mid I \in \I \text{ is strongly connected}\};
$$
\item
has as clusters the sets of $n$ cluster variables of the form
$$
\{X_{i_1},X_{i_2},\ldots,X_{i_k}\} \cup \{Y_S \mid S \in \s\}
$$
where $\s$ is a maximal nested collection for $[n] \setminus \{i_1,i_2,\ldots,i_k\}$;
\item
has cluster complex the extended nested complex of $\Gamma$.
\end{itemize}
\end{theorem}

\begin{example}\label{ex:first}
 Consider a graph $\Gamma$ on $[4]=\{1,2,3,4\}$ with edges $1 \longrightarrow 2$, $2 \longrightarrow 1$, $1 \longrightarrow 3$, $3 \longrightarrow 1$, $3 \longrightarrow 2$, $2 \longrightarrow 3$, $1 \longrightarrow 4$, $3 \longrightarrow 4$, $4 \longrightarrow 2$, as shown in Figure \ref{fig:lin1}.  Then the initial seed is given by 
$$t_\Gamma = \{(X_1, A_1 + X_2 + X_3 + X_4),(X_2, A_2 + X_1 +X_3),(X_3, A_3 + X_1 + X_2 + X_4),(X_4, A_4 + X_2)\}.$$
The graph $\Gamma$ has $11$ strongly-connected subsets: $\emptyset,\{1,4\},\{2,4\},\{3,4\},\{1,3,4\}$ are the subsets which are \emph{not} strongly-connected.  Thus $\A_\Gamma$ has $15$ cluster variables
$$
X_1,X_2,X_3,X_4,Y_1,Y_2,Y_3,Y_4,Y_{12},Y_{23},Y_{13},Y_{123},Y_{124},Y_{234},Y_{1234}
$$
It also has $46$ clusters, including:
$$
\{X_1,X_2,X_3,X_4\}, \{X_1,Y_2,X_3,X_4\},\{X_1,Y_2,X_3,Y_4\},\{Y_{124},Y_2,X_3,Y_4\},\{Y_{124},Y_2,Y_{1234},Y_4\},\ldots
$$
%\fixit{Expand this example to give some exchange relations?}
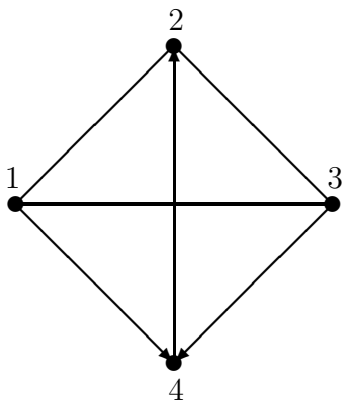
\begin{figure}[ht]
    \begin{center}
\setlength{\unitlength}{2pt}
\begin{picture}(60,60)(0,0)
\thicklines
\put(0,30){\line(1,1){30}}
\put(0,30){\vector(1,-1){30}}
\put(0,30){\line(1,0){60}}
\put(30,60){\line(1,-1){30}}
\put(60,30){\vector(-1,-1){30}}
\put(30,0){\vector(0,1){60}}
\put(0,30){\circle*{3}}
\put(-2,33){$1$}
\put(30,60){\circle*{3}}
\put(29,63){$2$}
\put(60,30){\circle*{3}}
\put(59,33){$3$}
\put(30,0){\circle*{3}}
\put(29,-7){$4$}
\end{picture}
\qquad
    \end{center} 
    \caption{The example graph
}
    \label{fig:lin1}
\end{figure}
\end{example}

\begin{remark}
It follows from the Laurent phenomenon property (Theorem \ref{thm:LP}) that each cluster variable $X_i$ or $Y_I$ is a Laurent polynomial when expressed in terms of the variables of any cluster.
\end{remark}

\subsection{Exchange relations}
The exchange polynomials $F_i^\s$ of the seed $t_\s$ associated to a nested collection $\s$ are described in Section \ref{sec:main}.  
Roughly speaking, the exchange polynomials count certain paths in $\Gamma$, when expressed not in terms of the cluster variables of $t_\s$, but in terms of all the $Y_I$.  
The exchange relation turns out to have an interpretation as a determinantal identity known as {\it Dodgson's condensation}.

\begin{example} 
In the example in Figure \ref{fig:lin1} take the cluster $\{Y_1,Y_{12},Y_{123},X_4\}$. The exchange relations in this cluster are  
\begin{align*}
 Y_1 Y_2 &= 1+Y_{12}, \;\;\;\;\;\;\;Y_{12} Y_{13} = 1+Y_1^2+Y_1(2+Y_{123}),
\end{align*}

\begin{align*}
Y_{123} X_3 &= \frac{X_4(1+Y_1)(1+Y_{12})+A_1(1+Y_1+Y_{12})+A_2 Y_1(1+Y_1)+A_3 Y_1 Y_{12}}{Y_1} \\
&=X_4(Y_{12}+Y_2 +1)+A_1(1+Y_2)+A_2(1+Y_1)+A_3Y_{12},
\end{align*}
and 
\begin{align*}
X_4 Y_{1234} &= \frac{1}{Y_1Y_{12}} (A_1(1+Y_{12}+Y_1^2+Y_1(2+Y_{123})+Y_1Y_{12}) \\
&+A_2(Y_1^3+Y_1^2(2+Y_{123})+Y_1)+A_3(Y_1^2Y_{12}+Y_1Y_{12})+A_4Y_1Y_{12}Y_{123})\\
&=A_1(Y_3+1)+A_2 Y_{13}+A_3(Y_1+1)+A_4Y_{123}
\end{align*}
\end{example}

\subsection{Graph associahedra as cluster complexes of LP algebras}\label{ss:graphassoc}
%Let $\T \subset 2^{[n]}$ be the collection of subsets which are strongly connected components of $\Gamma$.  

Let $\A'_\Gamma$ be the LP algebra obtained from $\A_\Gamma$ by freezing (see \cite{LP}) all the variables $\{Y_T \mid T \in \T\}$.  In more detail, consider the subset of seeds of $\A_\Gamma$ which contain all the cluster variables $\{Y_T \mid T \in \T\}$.  Removing all the $Y_T$-s from these seeds, we obtain the LP algebra $\A'_\Gamma$, where now the coefficient ring is $R[Y_T^{\pm 1} \mid T \in \T]$, but the ambient field $\FF$ is the same.  Note that none of the $X_i$ is a cluster variable in $\A'_\Gamma$, and the rank of $\A'_\Gamma$ is equal to $n - |\T|$.  

As a consequence of Theorem \ref{thm:AGamma}, we have

\begin{corollary}
The cluster complex of $\A'_\Gamma$ is the nested set complex of $\Gamma$.  The exchange graph of $\A'_\Gamma$ is the one-skeleton of the (di)graph associahedron $P(\Gamma)$ associated to $\Gamma$.
\end{corollary}

In the case that $\Gamma$ is an undirected path, $P(\Gamma)$ is the usual {\it associahedron}.  As explained in Section \ref{sec:examples}, in this case $\A'_\Gamma$ is in fact identical to the cluster algebra of type $A_n$, which is known to have cluster complex dual to the associahedron.

In the case that $\Gamma$ is a cycle, $P(\Gamma)$ is the {\it cyclohedron}.  However, $\A'_\Gamma$ is not the cluster algebra of type $B_n$, though it does have the same exchange graph.

We conjecture (Conjecture \ref{con:polytope}) that cluster complexes of all linear LP algebras are polytopal.

\subsection{General linear LP algebras}The exchange polynomials $F_i = A_i + \sum_{i \to j} X_j$ of our initial seed are not the most general linear polynomials that we could choose for a seed of a LP algebra.  Let us call a LP algebra $\A$ a {\it linear LP algebra} if it has a seed such that all exchange polynomials are linear in the cluster variables of that seed.   We conjecture (Conjecture \ref{con:finite}) that all linear LP algebras are of finite type.

If the $A_i$ are not algebraically independent, but satisfy relations, or for example are set to 0, then the corresponding linear LP algebra appears to be a degeneration of the graph LP algebra, and in particular has fewer clusters.  

Suppose the coefficients of $X_j$ in $F_i$ of a linear seed of a linear LP algebra $\A$ are arbitrary elements of a coefficient ring $R$, and not just equal to $0$ or $1$, but the constant term is still an indeterminate $A_i$.  One can still associate a graph $\Gamma$ which has an edge $i \to j$ if $X_j$ occurs in $F_i$ with non-zero coefficient.  The cluster complex of $\A$ appears to be the same as that of $\A_\Gamma$ studied here, though the coefficients of the exchange polynomials exhibit interesting dynamics.  This phenomenon is illustrated in the classification of rank 2 LP algebras in \cite{LP}.  The most interesting rank 2 linear case is that associated to the complete undirected graph on 2 vertices, with a pentagonal exchange graph.

\subsection{Positivity and other conjectures}
The LP algebras $\A_\Gamma$ appear to satisfy many positivity and integrality properties beyond the formulae proven in this paper.  We list some conjectures in Section \ref{sec:conjectures}.

\section{Background on Laurent phenomenon algebras}
\label{sec:background}
We recall background on LP algebras from \cite{LP}.

\subsection{Seeds}
\label{sec:seeds}

Let $R$ be a {\it coefficient ring} over $\Z$, which we assume to be a unique factorization domain.  For example $R$ could be $\Z$, a polynomial ring over $\Z$, or a Laurent polynomial ring over $\Z$.  Let $n \geq 1$ be a positive integer and write $[n]$ for $\{1,2,\ldots,n\}$.  Let the {\it ambient field} $\FF$ be the rational function field in $n$ independent variables over the field of fractions $\Frac(R)$. 

A {\it seed} in $\FF$ is a pair $(\x,\F)$ where
\begin{itemize}
\item
$\x = \{x_1,x_2,\ldots,x_n\}$ is a transcendence basis for $\FF$ over $\Frac(R)$.
\item
$\F = \{F_1,F_2,\ldots,F_n\}$ is a collection of polynomials in $\P = R[x_1,x_2,\ldots,x_n]$ satisfying:
\begin{enumerate}
\item[(LP1)]
$F_i$ is an irreducible element of $\P$ and is not divisible by any variable $x_j$
\item[(LP2)]
$F_i$ does not involve the variable $x_i$
\end{enumerate}
\end{itemize}
The variables $\{x_1,x_2,\ldots,x_n\}$ are called {\it cluster variables}, and the polynomials $$\{F_1,F_2,\ldots,F_n\}$$ are called {\it exchange polynomials}.  As is usual in the theory of cluster algebras, the set $\{x_1,x_2,\ldots,x_n\}$ will be called a {\it cluster}.
If $t = (\x,\F)$ is a seed, we let $\L = \L(t)$ denote the Laurent polynomial ring $R[x_1^{\pm 1},x_2^{\pm 1},\ldots,x_n^{\pm 1}]$.  
If $x$ is a cluster variable, we shall use the notation $F_x$ to denote the exchange polynomial associated to a cluster variable $x$.  
We call $n$ the {\it rank} of the seed $(\x,\F)$.  The set $\{x_1,\ldots,x_n\}$ of cluster variables in a seed will be called a {\it cluster}.

For each seed $(\x,\F)$, we define a collection $\{\hF_1,\hF_2,\ldots,\hF_n\} \subset \L$ of Laurent polynomials by the conditions:
\begin{itemize}
\item
$\hF_j = x_1^{a_1}\cdots \widehat{x_j} \cdots x_n^{a_n} F_j$ for some $a_1,\ldots,a_{j-1},a_{j+1},\ldots,a_n \in \Z_{\leq 0}$
\item
\begin{equation}\label{hatF}
\hF_i|_{x_j \leftarrow F_j/x} \in R[x_1^{\pm 1},\ldots,x_{j-1}^{\pm 1},x^{\pm 1},x_{j+1}^{\pm 1},\ldots,x_n^{\pm 1}] \; \mbox{and is not divisible by $F_j$}
\end{equation}
The divisibility is to be checked in $R[x_1^{\pm 1},\ldots,x_{j-1}^{\pm 1},x^{\pm 1},x_{j+1}^{\pm 1},\ldots,x_n^{\pm 1}]$.
\end{itemize}
The $\hF$ are well-defined (\cite[Lemma 2.3]{LP}). Furthermore, the collections $\{F_1,\ldots,F_n\}$ and $\{\hF_1,\ldots,\hF_n\}$ determine each other uniquely (\cite[Lemma 2.3]{LP}).

Suppose $i \in [n]$.  Then we say that a tuple $(\x',\F')$ is obtained from a seed $(\x,\F)$ by mutation at $i$, and write $(\x',\F')  = \mu_i(\x,\F)$, if the former can be obtained from the latter by the following procedure.  

The cluster variables of $\mu_i(\x,\F)$ are given by $x'_i = \hF_i/x_i$ and $x'_j = x_j$ for $j \neq i$.  The exchange polynomials $F'_j \in \L'$ are obtained from $F_j$ as follows.  
If $F_j$ does not depend on $x_i$, then we must have $F'_j/F_j$ be a unit in $R$, where now $F'_j$ is considered as an element of $\L'$.  Otherwise $\hF_i(0)$ is well defined.  We define $G_j$ by
\begin{equation}\label{E:G}
G_j = F_j|_{x_i \leftarrow \frac{\hF_i|_{x_j \leftarrow 0}}{x'_i}}
\end{equation}
Next, we define $H_j$ to be $G_j$ with all common factors (in $R[x_1,\ldots,\hat x_i,\ldots,\hat x_j \ldots,x_n]$) with $\hF_i|_{x_j \leftarrow 0}$ removed.  Note that this defines $H_j$ only up to a unit in $R$.  Finally we have $F'_j = MH_j$ where $M$ is a Laurent monomial in the $x'_1,x'_2,\ldots,\widehat{x'_j},\ldots, x'_n$ with coefficient in $R$, such that $F'_j \in \P'$, satisfies $(LP2)$, and is not divisible by any variable in $\P'$.  For any $H_j$, it is always possible to pick the monomial $M$ to satisfy these conditions, but in general there are many choices for the coefficient of $M$.  In particular $F'_j$ is defined only up to a unit in $R$.

In \cite[Proposition 2.9]{LP} it was shown that if $(\x',\F') = \mu_i(\x,\F)$ is obtained by mutation of $(\x,\F)$ at $i$ then $(\x',\F')$ is also a seed.  
It is clear that if $\x$ is a transcendence basis of $\FF$ over $\Frac(R)$, then so is $\x'$. It was shown in \cite[Proposition 2.10]{LP} that mutations are involutions.
\begin{remark}\label{rem:subs}
By definition, the mutated exchange polynomial $F'_i$ is obtained from $F_i$ by the substitution $x_k \leftarrow \hF_k|_{x_i = 0}/x'_k$.  We will often use the fact that sometimes we can perform this substitution in $\hF_i$ instead of $F_i$, with the answer only being off a monomial factor.  This is the case when $\hF_i/F_i$ does not involve $x_k$.
\end{remark}

\subsection{Laurent phenomenon algebras}
Let $R$ be a fixed coefficient ring and $\FF$ denote the ambient fraction field in $n$ indeterminates as in Section \ref{sec:seeds}.  A {\it Laurent phenomenon algebra} $(\A, \{(\x,\F\})$ is a subring of $\A \subset \FF$ together with a distinguished collection of seeds $\{(\x,\F)\} \subset \FF$ belonging to the ambient field $\FF$.  The algebra $\A \subset \FF$ is generated over $R$ by all the variables $\x$ in any of the seeds of $\A$.  The seeds satisfy the condition: for each seed $(\x,\F)$ and $i \in [n]$, we are given a seed $(\x',\F') = \mu_i(\x,\F)$ obtained from $(\x,\F)$ by mutation at $i$. Thus the seeds form the vertices of a $n$-regular graph, where the edges are mutations.  Furthermore, we assume all seeds are connected by mutation.  

If $t = (\x,\F)$ is any seed in $\FF$, we shall let $\A(t)$ denote any LP algebra which has $t$ as a seed.  We say that $\A(t)$ is generated (as a LP algebra) by $t$, or has {\it initial seed} $t$.  Since seed mutation is only well-defined up to units, the seeds of $\A(t)$ are not determined by $t$.  

We say that two seeds $(\x,\F)$ and $(\x',\F')$ are {\it equivalent} if the following two conditions hold:
\begin{enumerate}
\item
For each $i$ we have $x_i/x'_i$ is a unit in $R$, and
\item
For each $i$ we have $F_i/F'_i$ is a unit in $R$, where $F_i$ and $F'_i$ are both considered as elements of the ambient field $$\FF = \Frac(R[x_1,x_2,\ldots,x_n]) = \Frac(R[x'_1,x'_2,\ldots,x'_n]).$$
\end{enumerate}

Let $\A$ be a Laurent phenomenon algebra.  We will say that $\A$ is {\it normalized} if whenever two seeds $t_1$, $t_2$ are equivalent, we have that $t_1 = t_2$.  Suppose $\A$ is a LP algebra which is normalized.  Then we say that $\A$ is of {\it finite type} if it has finitely many seeds.  All the LP algebras we study in this paper will be normalized and of finite type.

\subsection{Some properties of LP algebras}
\begin{theorem}{\cite[Theorem 5.1]{LP}}\label{thm:LP}
The Laurent phenomenon holds: every cluster variable of a LP algebra $\A$ is a Laurent polynomial in the cluster variables of any seed $t$ of $\A$.
\end{theorem}

\begin{lemma} \label{lem:monden}
Let $t = (\x,\F)$ be a seed of a LP algebra $\A(t)$, and $x_i$ one of the cluster variables of $t$.  Suppose that $m = \prod_{j \neq i} x_j^{d_j}$ is a monomial in other cluster variables of $t$ such that $F_i/m$ can be expressed as a {\it {polynomial}} in the cluster variables (not necessarily belonging to $t$) of $\A$. 
Suppose $F_m/\hF_m = \prod_{j\neq i} x_j^{d'_j}$.  Then we have $d'_j \geq d_j$. 
\end{lemma}

\begin{proof}
By Theorem \ref{thm:LP}, being a polynomial in some cluster variables, $F_i/m$ must be a Laurent polynomial in any cluster of this LP algebra.  This implies that $F_i/m \in \L(t) \cap \L(\mu_j(t))$.  Comparing with \eqref{hatF}, we obtain the desired inequality.
\end{proof}

\begin{theorem} \label{thm:comm}
Let $t = (\x,\F)$ be a seed and $i \neq j$ be two indices be such that $x_j$ does not appear in $F_i$ and such that $F_i/F_j$ is not a unit in $R$.  Then the mutations at $i$ and $j$ commute.  More precisely, we can choose seed mutations so that 
$$
\mu_i(\mu_j(t)) = \mu_j(\mu_i(t)).
$$
\end{theorem}
%\fixit{Say what exactly commute means}

The proof of Theorem \ref{thm:comm} is postponed until Section \ref{sec:proof4}.

\subsection{Notation}

For $f,g \in \L(t)$ we shall write $f \propto g$ to mean that $f$ and $g$ are equal up to a unit in $\L(t)$.  That is, $f$ and $g$ differ multiplicatively by a (Laurent) monomial factor $r\prod_i x_i^{a_i}$ in the cluster variables of $t$, where $r \in R$ is a unit in $R$.  Note that this differs from the usage of $\propto$ in \cite{LP}.

\section{Nested complexes}
For a set $S$ and elements $i,j,k$ (usually not in $S$), we will use the concatenations $Si, Sij, Sijk$ to denote the unions $S \cup\{i\}$, $S \cup \{i,j\}$ and $S \cup\{i,j,k\}$.  Occasionally we will also concatenate subsets and elements such as $SiTj$ to indicate the union $S \cup \{i\} \cup T \{j\}$.

\subsection{Maximal nested collections}

Let $\Gamma$ be a directed loopless multiplicity-free graph on the vertex set $[n] = \{1,2,\ldots,n\}$. Thus, every edge $i \longrightarrow j$ is either present with multiplicity one or absent, for each ordered pair $(i,j)$, $i \not = j$. We identify subsets of vertices of $\Gamma$ with the corresponding induced subgraphs, for example we shall talk about {\it {strongly connected}} subsets, and so on.  Recall that we defined nested collections $\s = \{S_1,S_2,\ldots,S_k\} \subset \I$ in Section \ref{ss:nested}.  We say that a nested collection $\s$ is {\it maximal} for $S = \bigcup_i S_i$ if for any other nested collection $\s'$ with $\bigcup_i S'_i = S$ we have $\s \subset \s'$ implies $\s = \s'$.  A nested collection $\s$ is maximal if it is maximal for its support.

%Let $\I \subset 2^{[n]}$ denote the collection of strongly-connected subsets of $\Gamma$. A family of subsets $\{I_1, \ldots, I_k\} \in \I$ is {\it nested} if 
%\begin{itemize}
% \item for any pair $I_i, I_j$ either one of them lies inside the other, or they are disjoint;
% \item for any tuple of disjoint $I_j$-s, they are the strongly connected components of their union.
%\end{itemize}

\begin{example} \label{ex:nested}
 Consider the graph in Example \ref{ex:first}.
%a graph on four vertices with edges $1 \longrightarrow 2$, $2 \longrightarrow 1$, $1 \longrightarrow 3$, $3 \longrightarrow 1$, $3 \longrightarrow 2$, $2 \longrightarrow 3$, $1 \longrightarrow 4$, $3 \longrightarrow 4$, $4 \longrightarrow 2$.
The following are some of the families of nested subsets:
$$\{\{3\},\{4\},\{2,3,4\},\{1,2,3,4\}\},$$
$$\{\{3\},\{4\},\{1,3\},\{1,2,3,4\}\},$$
$$\{\{1\},\{1,2\},\{1,2,3\},\{1,2,3,4\}\}.$$
All of these families are maximal.
\end{example}

The following result can for example be found in \cite[Proposition 7.6]{Pos}. 

\begin{lemma}\label{lem:nested}
Suppose $\s=\{S_i\}_{i=1}^k$ is a maximal nested collection for $S = \bigcup_i S_i$.  Then there is a bijection $\phi: S \to \s$ given by the condition that $\phi(s)$ is the smallest subset in $\s$ containing $s$.  
\end{lemma}

 We will always assume that the subsets are indexed so that the bijection $\phi$ is the identity: that is, $S_i$ is the smallest subset containing $i$.  Lemma \ref{lem:nested} gives rise to a partial order $\preceq_\s$ on $S$: we have $i \preceq_\s j$ if and only if $S_i \subseteq S_j$.  Furthermore, this partial order uniquely determines $\s$.  For any non-maximal $i \in S$, there is a unique $j \in S$ which covers $i$ in this partial order.  We denote this element by $i^+$.  We denote the maximal elements of $S$ by $S^{\max}$.

Let $\s$ be a maximal nested collection.  Suppose $k \in S^{\max}$ is maximal.  Then removing $S_{k}$ from $\s$ gives a legitimate nested collection $\s^{k}$, which is a maximal nested collection for $S - \{k\}$.  Also let $k \in S$ be arbitrary.  Then define $\s_k$ by removing from $\s$ all subsets $S \supsetneq S_k$.  If $k \in S^{\max}$, then $\s_k = \s$.

 For $S \subset [n]$ denote by $S \oplus j$ the strongly connected component of $Sj$ that contains $j$. Denote by $S \ominus j = Sj - (S \oplus j)$ the elements of $S$ that are not included in $S \oplus j$.  Note that these definitions are made both for $j \in S$ and for $j \not \in S$.

\begin{example}
 Consider the maximal nested collection $\s = \{3\},\{4\},\{2,3,4\},\{1,2,3,4\}$ from Example \ref{ex:nested}. Then 
$$\phi(1) = \{1,2,3,4\}, \;\; \phi(2) = \{2,3,4\}, \;\; \phi(3) = \{3\}, \;\; \phi(4) = \{4\}.$$
So for the indexing associated with $\phi$ we have $S_2 = \{2,3,4\}$, and so on. The corresponding partial order is given by transitive closure of the relation $3 \preceq_\s 2$, the relation $4 \preceq_\s 2$ and the relation $2 \preceq_\s 1$. We also have 
$3^+ = 4^+ = 2$, $2^+ = 1$ and $\s_2 = \{3\},\{4\},\{2,3,4\}$, $\s_3 = \{3\},\{4\}$. Finally, we have $\{3,4\} \oplus 1 = \{1,3\}$, $\{3,4\} \ominus 1 = \{4\}$, $\{3,4\} \oplus 2 = \{2,3,4\}$, $\{3,4\} \ominus 2 = \emptyset$.
\end{example}

\subsection{Mutation of nested collections}
\label{sec:nestmutations}
We describe a way of building maximal nested collections from the empty nested collection.  Suppose $\s$ is a maximal nested collection with support $S$ and $s \notin S$.  We define $\mu_s(\s) = \s \cup \{S \oplus s\}$, which has support $S \cup \{s\}$.  This is a kind of {\it mutation} of $\s$ which we call {\it activation} at $s$.  Given an ordered subset $\vec{S} = (s_1,s_2,\ldots,s_k) \subset [n]$, we define $\s_0,\s_1,\ldots,\s_k = \s$ as follows.  We have $\s_0 = \emptyset$ and recursively define  $\s_{r+1} =\mu_{s_{r+1}}(\s_r)$. 

\begin{lemma}\label{lem:poset}
The activation $\mu_s(\s)$ for $s \notin S$ is a  maximal nested collection on $S \cup \{s\}$.  Any maximal nested collection can be obtained by activating a sequence $\vec{S} = (s_1,s_2,\ldots,s_k)$.
\end{lemma}

\begin{proof}
The fact that $\mu_s(S)$ is nested is checked directly from the definition, and maximality follows from counting.  For the second statement, recall that maximal nested collections $\s$ have a partial order $\preceq_\s$, given by the ordering $i \preceq_\s j$ if $S_i \subset S_j$. Now, any sequence $(s_1,s_2,\ldots,s_k)$ of activations that agree with the partial order, that is, satisfies $s_r \preceq_\s s_{r'} \implies r \leq r'$, yield the desired $\s$.
\end{proof}

\begin{example}
 Consider the activation sequence $(3,4,2,1)$ for the graph in Example \ref{ex:first}. Using
$$\emptyset \oplus 3 = \{3\}, \;\; \{3\} \oplus 4 = \{4\}, \;\; \{3,4\} \oplus 2 = \{2,3,4\}, \;\; \{2,3,4\} \oplus 1 = \{1,2,3,4\}$$
we see that the result is the maximal nested collection $\{\{3\},\{4\},\{2,3,4\},\{1,2,3,4\}\}$. If on the other hand we used the activation sequence $(3,4,1,2)$, we would end up with $\{3\},\{4\},\{1,3\},\{1,2,3,4\}$.
\end{example}

Let $\vec{S} = (s_1, \ldots, s_k)$ be a sequence of activations. We call a pair of activations $s_i$ and $s_{i+1}$ {\it {exchangeable}} if $S_{s_i}$ and $S_{s_{i+1}}$ are disjoint. 

\begin{lemma}\label{lem:exc}
If we swap an exchangeable pair of activations in $\vec{S}$, the resulting collection $\s$ does not change. Furthermore, any two sequences of activations that yield the same $\s$ can be transformed one into another by swapping exchangeable pairs.
\end{lemma}

\begin{proof}
 The first claim is clear, since if $s_i$ and $s_{i+1}$ are not strongly connected via $(s_1, \ldots, s_{i-1})$, it does not matter which one of the two we activate first. For the second claim, observe that Lemma \ref{lem:poset} gives a bijection between activation sequences yielding $\s$ and linear extensions of $\preceq_\s$.  Thus our statement is equivalent to the following combinatorial statement: any two linear extensions of a poset $P$ can be connected by successively swapping the positions of $i$ and $i+1$, while always staying a linear extension.  This is a simple combinatorial fact that can easily be verified.
\end{proof}

So far we have discussed activating an element (or vertex) $s \notin S$.  Now we describe the mutation $\mu_s(\s)$ for $s \in S$.  If $s \in \s^{\max}$ then $\mu_s(\s) = \s^s$ is obtained from $\s$ by removing $S_s$.  In this case we can say that $s$ has been {\it deactivated} and $\mu_s(\s)$ is a maximal nested collection with support $\s - \{s\}$.  If $s \notin s^{\max}$, we define $\mu_s(\s)$ to be the collection obtained from $\s$ by removing $S_s$ and adding the new subset $S_{s^+}-s \oplus s^+$.  We call this operation {\it internal mutation} at $s$.

\begin{example}
 Take the maximal collection $\{3\},\{4\},\{2,3,4\},\{1,2,3,4\}$, and take $s = 2$. Then $2^+ = 1,$ $S_1 - 2 = \{1,2,3,4\}-2 = \{1,3,4\}$ and $\{1,3,4\} \oplus 1 = \{1,3\}$. Thus, we remove $S_2 = \{2,3,4\}$ from the family, and replace it by $\{1,3\}$, thus obtaining the new collection $\{3\},\{4\},\{1,3\}$,$\{1,2,3,4\}$.
\end{example}

\begin{lemma}
The internal mutation $\mu_s(\s)$ for $s \in S - S^{\max}$ is a maximal nested collection on $S$.  Furthermore, $\mu_s(\s)$ is the unique maximal nested collection with the same support as $\s$, other than $\s$, which contains all the subsets $\{S_r \mid r \neq s\}$.
\end{lemma}
\begin{proof}
Maximality and the claim about the support of $\s' = \mu_s(\s)$ are clear.  We check that $\s'$ is nested.  First note that we have $S'_s = S_{s^+}$ and $S'_{s^+} = S_{s^+}-s \oplus s^+$.

Suppose $S_r \in \s'$ and $S'_{s^+} = S_{s^+}-s \oplus s^+$ intersect non-trivially.  Then $S_r$ intersects $S_{s^+}$ non-trivially, and so must be contained in $S_{s^+}$.  But then it is clear from the definition of $S'_{s^+}$ that either $S_r \subset S'_{s^+}$ or $s \in S_r$.  In the latter case we must have $S_r = S_{s^+}$.

Now suppose we have a disjoint collection $S_{r_1},S_{r_2},\ldots,S_{r_k},S'_{s^+} \in \s'$.  We need to show that they are the strongly connected components of their union.  The only possible way this can fail is if we can find $i_1,i_2,\ldots,i_j$ so that $S_{r_{i_1}},S_{r_{i_2}},\ldots,S_{r_{i_j}}$ and $S'_{s^+}$ have a strongly connected union which is a subset of $S_{s^+}$.  This is impossible since none of $S_{r_{i_1}},S_{r_{i_2}},\ldots,S_{r_{i_j}}$ contain $s$, and the assumption that $\s$ was nested.

For the last statement, suppose $\s'$ is a maximal nested collection containing all the subsets $\{S_r \mid r \neq s\}$.  For any $r \neq s,s^+$ it is clear that $\s'_r = \s_r$, so by Lemma \ref{lem:nested} we have that $\s'$ is determined by whether we have $\s'_s = \s_s$ or $\s'_s = \s_{s^+}$.  In the first case we have $\s' = \s$, and the in latter case we have $\s' = \mu_s(\s)$.
\end{proof}

\subsection{The exchange graph on maximal nested collections}\label{sec:exc}
Let $\C$ denote the set of all maximal nested collections $\s$ with support equal to any subset $S \subset [n]$.  We define a $n$-regular graph with vertex set $\C$ by connecting $S$ and $\mu_s(\s)$ by an edge labeled $s$.  Note that we never have $\mu_s(\s) = \s$.

\section{Acyclic functions}
\label{sec:Y}
\subsection{Acyclic functions and strongly connected components}
We first give an example of the $Y$ Laurent polynomial defined in Section \ref{ss:Y}.
\begin{example}\label{ex:Y}
Consider the graph $\Gamma$ of Example \ref{ex:first}.  Then we have 
$$Y_{124} = \frac{X_1(X_2(X_3+A_1)+A_4(X_3+X_4+A_1)) + (X_2+A_4)(X_2+X_3+X_4+A_1)(X_3+A_2)}{X_1 X_2 X_4}.$$
For example, the term $X_2 X_3 A_2$ corresponds to the function $f(1) = 3$, $f(2)=2$, $f(4)=2$. The functions $f$ that do {\it {not}} contribute to the numerator are the functions that have either $f(1)=2, f(2)=1$ (and $f(4) = 4 \; \text{or} \; f(4)=2$), or the function $f(1)=4, f(4)=2, f(2)=1$.
\end{example}

We continue to fix a multiplicity-free loopless directed graph $\Gamma$ on $[n]$. 

\begin{lemma} \label{lem:Yfac}
 Let $I_1, \ldots, I_k \in \I$ be the strongly connected components of $I$. Then $Y_I = \prod_{j=1}^k Y_{I_j}$.
\end{lemma}

\begin{proof}
Any combination of acyclic functions on the $I_j$ yields an acyclic function on $I$, since there is no cycle that can be created in between different strongly connected components of $I$.
\end{proof}

\begin{lemma}\label{lem:YiS}
 We have $Y_{Si} = Y_{S \oplus i} Y_{S \ominus i}$.
\end{lemma}

\begin{proof}
 Clear from Lemma \ref{lem:Yfac}.
\end{proof}

\begin{lemma} \label{lem:irr}
Suppose $I \in \I$ is strongly connected.  Then the numerator of $Y_I$ is irreducible in the ring of polynomials in $X_i$.  Furthermore, these numerators are pairwise coprime.
\end{lemma}

\begin{proof}
The proof is by induction on the size of $I$, and the base case $|I|=1$ is clear. Assume the numerator of $Y_I$ factors as $PQ$. Since it is multilinear in the $A_i$, $i \in I$, so must be each of the factors $P$ and $Q$. Furthermore, each $A_i$ appears either in $P$ or in $Q$ but not both. Let $I_P$ and $I_Q$ be the sets of $i$-s such that $A_i$ appears in the corresponding factor. Let $j \in I_P$. Then since numerator of $Y_I$ is linear in $A_j$, the coefficient of $A_j$ must be divisible by $Q$. However, this coefficient is the numerator of $Y_{I - \{j\}}$. By induction assumption, and from the fact that the polynomial ring is a unique factorization domain, it follows that $Q$ must be product of numerators of $Y_J$ for some strongly connected components $J$ of $I - \{j\}$. Since this is true for any $j \in I_P$, and since each $A_i$, $i \not \in I_P$ must appear in $Q$, we conclude that $Q$ is the numerator of $Y_{I_Q}$. Similarly $P$ is the numerator of $Y_{I_P}$. However, the identity $Y_I = Y_{I_P} Y_{
I_Q}$ fails, since $I$ is strongly connected and thus there are cycles that fail to be cycles when restricted to $I_P$ or $I_Q$.

The last statement is easy: the numerator of $Y_I$ has a monomial $\prod_{i\in I} A_i$ which occurs in no other numerator.
\end{proof}

\subsection{Acyclic functions as determinants}

We shall use the notation $Y_i := Y_{\{i\}}$. Let $\Y = (\y_{ij})$ be the $n \times n$ matrix defined by 
$$
\y_{ij} = 
\begin{cases}
Y_{i} & \text{if $i=j$;}\\
-1 & \text{if there is an edge $i \longrightarrow j$ in $\Gamma$;}\\
0 & \text{otherwise.}
\end{cases}
$$
Denote $\Y_I$ the principal minor of $\Y$ with row and column indices $I$. 
\begin{prop} \label{prop:det}
We have
$$Y_I = \det(\Y_I).$$
\end{prop}
\begin{proof}
 By the inclusion-exclusion principle, to count the acyclic functions from the set $I$ (with weight assigned as above) we need to take the alternating sum over all collections of disjoint cycles in $I$ of functions containing those cycles. In other words,
$$Y_I = \sum_{\text{collections $\C$ of disjoint cycles}} (-1)^{\text{number of cycles in $\C$}} \; Y_{I - \C}.$$ 
This is exactly the expansion of the determinant of the minor $\Y_I$. Indeed, the off-diagonal entries give a collection of disjoint cycles in $\Gamma$. Furthermore, the sign of a term in the determinant is $(-1)^{\text{number of even cycles}}$, while the sign obtained from having $-1$ entries off diagonal in $\Y$ is $(-1)^{\text{number of odd cycles}}$. Together, those give us the desired sign $(-1)^{\text{number of cycles}}$.
\end{proof}

\begin{example}
For the graph of Example \ref{ex:first}, we have 
$$
\Y =  \left(\begin{matrix}
\frac{A_1 + X_2 + X_3 + X_4}{X_1} &  -1  & -1 & -1\\
-1  &  \frac{A_2+X_1+X_3}{X_2} & -1 & 0\\
-1 & -1 & \frac{A_3+X_1+X_2+X_4}{X_3} & -1\\
0  &   -1       & 0  & \frac{A_4 + X_2}{X_4}
\end{matrix}\right).
$$
It is easy to check that the variable $Y_{124}$ in the example above is equal to the principal minor of $\Y$ with row and column set $\{1,2,4\}$.
\end{example}

\subsection{Exchange relations}

We describe what will turn out to be the exchange relations in the LP algebra associated with $\Gamma$. Let us use the notation $p: i \longrightarrow_S j$ to denote simple (vertex non-repeating) paths $p$ from vertex $i$ to vertex $j$ such that all intermediate vertices along the path belong to the set $S$. The $i$ and $j$ vertices may or may not belong to $S$. Define
$$P_S^{i,j} = \sum_{p: i \longrightarrow_S j} Y_{S - p}.$$ Note that for $j=i$ the only simple path from $i$ to $i$ is trivial path, and thus $P_S^{i,i} = Y_S$ for $i \not \in S$ and $P_S^{i,i} = Y_{S - \{i\}}$ for $i \in S$.

\begin{lemma} \label{lem:maxmut}
For $i \not \in S$ we have 
$$X_i Y_{S \oplus i} = \frac{\sum_{j \not \in S i} P_S^{i,j} X_j + \sum_{j \in Si} P_S^{i,j} A_j}{Y_{S \ominus i}}.$$
\end{lemma}

\begin{proof}
By Lemma \ref{lem:YiS} the desired equality is equivalent to $$X_i Y_{Si} = \sum_{j \not \in Si} P_S^{i,j} X_j + \sum_{j \in Si} P_S^{i,j} A_j.$$ For any acyclic function counted in the numerator of $Y_{S i}$, just follow the arrows from $i$ until this path either leaves $S$ or end with a loop in $S$. Group acyclic functions by the endpoints $j$ of the resulting paths $p$.
\end{proof}

\begin{lemma} \label{lem:Fpoly}
 The expression 
$$\frac{\sum_{j \not \in Si} P_S^{i,j} X_j + \sum_{j \in Si} P_S^{i,j} A_j}{Y_{S \ominus i}}$$
can be written as a polynomial in the $X_j$, $A_j$ and $Y_T$'s for $T \subset S$.
\end{lemma}

\begin{proof}
 Since it is equal to $X_i Y_{S \oplus i}$, it depends only on $S \oplus i$ and not $S$ itself. Thus we can compute it taking Lemma \ref{lem:maxmut} with $S := S \oplus i$, in which case $S \ominus i = \emptyset$ so that there is no denominator.
\end{proof}

\begin{example}
Continuing Example \ref{ex:Y}, Lemma \ref{lem:maxmut} gives for $S = \{1,2\}$, $i=4$:
$$X_4 Y_{124} = \frac{X_3(Y_1+1)+A_4Y_{12}+A_1+A_2 Y_1}{1},$$
and for $S =\{3,4\}$ and $i=1$:
$$X_1 Y_{13} = \frac{X_2(Y_{3}Y_{4}+Y_3+Y_4+1)+A_1 Y_3 Y_4 +A_4(Y_3+1)+A_3 Y_4}{Y_4}.$$

This last expression can be rewritten
$$\frac{X_2(Y_{3}Y_{4}+Y_3+Y_4+1)+A_1 Y_3 Y_4 +A_4(Y_3+1)+A_3 Y_4}{Y_4} = \frac{X_2(Y_3+1)+X_4(Y_3+1)+A_1Y_3+A_3}{1}.$$

\end{example}

\begin{lemma} \label{lem:det}
For $i \neq j$, and a set $S$ we have $$P_S^{i,j} = (-1)^{1+\text{number of elements of $S$ between $i$ and $j$}} \det(\Y_{S'i, S'j}),$$ where $S'=S -\{i,j\}$ and $\Y_{I,J}$ denotes the minor of $\Y$ with row indices $I$ and column indices $J$.
\end{lemma}

\begin{proof}
 Swap columns of $\Y_{Si, Sj}$ so that column $j$ moves to the $i$-th position. This creates a sign $(-1)^{\text{number of elements of $S$ between $i$ and $j$}}$. In the expansion of the resulting determinant there is one cycle which corresponds to a path from $i$ to $j$ in $S$. Group terms according to this path. According to Proposition \ref{prop:det} we get exactly the summation in the definition of $P_S^{i,j}$, except we need to check that the signs work out correctly. The terms corresponding to a specific path $p$ from $i$ to $j$ receive minus sign from the definition of the determinant in case the path has even length. They also receive minus sign from $-1$ entries of $\Y$ in case $p$ has odd length. Thus, all terms have the same minus sign, which contributes $1$ in the exponent of $(-1)^{1+\text{number of elements of $S$ between $i$ and $j$}}$.
\end{proof}

Recall that the {\it Jacobi identity}, also known as {\it Dodgson's condensation} states that for a square $n \times n$ matrix $M$ we have
\begin{equation}\label{eq:Jacobi}
\det(M) \, \det(M_{[2,n-1],[2,n-1]}) = \det(M_{[1,n-1],[1,n-1]})\,\det(M_{[2,n],[2,n]}) -\det(M_{[1,n-1],[2,n]}) \, \det(M_{[2,n],[1,n-1]}).
\end{equation}

\begin{lemma} \label{lem:Yex}
For $i, j \not \in S$ with $i \neq j$ we have 
$$Y_{S \oplus i} Y_{S \oplus j} = \frac{Y_{S ij} Y_{S} + P_S^{i,j} P_S^{j,i}}{Y_{S \ominus i} Y_{S \ominus j}}.$$
\end{lemma}

\begin{proof}
By Lemma \ref{lem:YiS}, the equality is equivalent to 
$$Y_{Si} Y_{Sj} = Y_{Sij} Y_{S} + P_S^{i,j} P_S^{j,i}.$$
According to Lemma \ref{lem:det}, this is just the identity \eqref{eq:Jacobi}.
\end{proof}

\begin{example}
With the graph in Example \ref{ex:first}, we have
$$P_{\{3,4\}}^{1,2} = Y_{34}+Y_3+Y_4+1 = 
- \det \left(\begin{matrix}
  -1  & -1 & -1\\
 -1 & Y_3 & -1\\
   -1       & 0  & Y_4
\end{matrix}\right)
$$
and $P_{\{3,4\}}^{2,1} = Y_{34}+Y_4$.  For $S = \{3,4\}$, $i=1$ and $j=2$, Lemma \ref{lem:Yex} gives
$$Y_{13} Y_{234}= \frac{Y_{1234} Y_{34} + (Y_{34}+Y_3+Y_4+1)(Y_{34}+Y_4)}{Y_4}.$$
\end{example}

\begin{lemma} \label{lem:ikj}
Let $i,k \notin S$, and $j$ be arbitrary such that $i \neq k$ and $j \neq k$.  Then
$$P_{Sk}^{i,j} Y_S = \begin{cases} P_S^{i,k} P_{S}^{k,j} + P_S^{i,j} Y_{Sk} & \mbox{if $i \neq j$} \\
P_S^{i,j} Y_{Sk} & \mbox{if $i = j$.}
\end{cases}
$$
\end{lemma}

\begin{proof}
When $i = j$, the claim follows from $P^{i,j}_{Sk} = Y_{Sk}$ and $P^{i,j}_S = Y_S$. When $i \neq j$ and $j \notin S$, this is a consequence of \eqref{eq:Jacobi} and Lemma \ref{lem:det} applied to the matrix $\Y_{Sik,Sjk}$.  One easily verifies that the sign works out correctly in all cases of relative positions of $i$, $j$ and $k$.
For $j \in S$, we can argue as follows: create a new ``virtual'' vertex $j'$ outside $S$, with a single edge from $j$ to $j'$.  Then the claimed identity is the same as that for $i,k,j'$.
\end{proof}

\subsection{Algebraic dependencies among the $Y$ variables}

Lemma \ref{lem:Yex} has the following consequence. 

\begin{proposition} \label{prop:Yrat}
Let $\s$ be any maximal nested collection with support $S \subset [n]$.  Any $Y_I$ for $I \subseteq S$ is a rational function (with integer coefficients) in the $Y_{S_i}$, $S_i \in \s$.
\end{proposition}

\begin{proof}
The proof is by induction on size of $S$. The claim is clearly true for $|S|=1$. Assume we have verified the claim for all the smaller sizes than the current one. 
It is clearly enough to consider the case when $S$ is connected, otherwise the claim follows recursively from knowing it for all connected components of $S$.  If $I = S$ the claim is clearly true, so we assume that $I \subsetneq S$.

We first argue that without loss of generality we can assume that $I$ does not contain the maximal element $j$ of $S$ in $\s$. Indeed, pick any element $k \in S - S^{\max}$ which is not in $I$.  Let $\s' = \mu_k(\s)$ be the maximal nested collection supported on $S$ obtained by exchanging $Y_{S_k}$ for $Y_{(S_{k^+} -\{k\}) \oplus k^+}$.  Apply Lemma \ref{lem:Yex} to the product $Y_{S_k}Y_{(S_{k^+} -\{k\}) \oplus k^+}$.  (In the notation of Lemma \ref{lem:Yex}, our $k$ and $k^+$ is $i$ and $j$ respectively.)  The right hand side of the exchange involves $Y_{S_{k^+}}$, which is one of the $Y_{S_i}$, $S_i \in \s$. All other $Y$-s it involves are supported on $S_{k^+}-k^+$, and thus by the induction assumption are rational functions of $Y_{S_i}$, $S_i \in \s^{k^+} \subset \s$.  We conclude that $Y_{(S_{k^+} -\{k\}) \oplus k^+}$ is a rational function in $\{Y_{S_i} \mid S_i \in \s\}$.
Thus it suffices to show $Y_I$ is a rational function of $\{Y_{S_i} \mid S_i \in \s'\}$. In $\s'$ the element $k$ is one step closer to being maximal. Proceeding in this manner, we shall arrive at a maximal nested collection on $S$ where $k \not \in I$ is maximal.

Now, since $i \not \in I$, we have $Y_I = \prod_j Y_{I \cap S_j}$, where the product is over $j$ maximal in $\s^i$. It is possible this product has just one term, but regardless of that we can now apply the induction assumption to conclude that $Y_I$ is rational function of $Y_{S_i}$, $S_i \in \s^i$.

\end{proof}

\begin{example}
 Let $S = \{1,2,3\}$ in the graph of Example \ref{ex:first}, and take $\s = \{3\},\{2,3\},\{1,2,3\}$. Then the expression for $Y_1$ in terms of this maximal nested collection is 
$$Y_1 = \frac{1+Y_3^2+Y_{23}+Y_3(2+Y_{123})}{Y_3 Y_{23}}.$$
\end{example}

\section{Graph LP algebras}
\label{sec:main}

\subsection{The LP algebra with linear initial seed}
Fix $\Gamma$ a directed, multiplicity-free, loopless graph on $[n]$.  We shall define a LP algebra with coefficient ring $R = \Z[A_1,A_2,\ldots,A_n]$, and initial cluster variables $X_1,X_2,\ldots,X_n$.  In fact the choice of the coefficient ring $R$ is not critical; the key property we need is that the $A_i$ are algebraically independent.

Let $\s \in \C$ be a maximal nested collection supported on the set $S \subseteq [n]$.  We now define a seed $t_\s = \{(X^{\s}_i,F_i^\s)\}$.  The cluster variables are defined by
\begin{equation}\label{E:cluster}
X^{\s}_i = \begin{cases}
X_i & \mbox{for $i \notin S$} \\
Y_{S_i} & \mbox{for $i \in S$.}
\end{cases}
\end{equation}
The indices, or cluster variables $X^{\s}_i$ for $i \notin S$ are called {\it external}, and for $i \in S$ are called {\it internal} or {\it activated}.  The exchange Laurent polynomials are defined by
\begin{equation}\label{E:hFi}
\hF_i^{\s} = \begin{cases}
\frac{\sum_{j \not \in Si} P_S^{i,j} X_j + \sum_{j \in Si} P_S^{i,j} A_j}{Y_{S \ominus i}} & \mbox{for $i \notin S$ or $i \in S^{\max}$} \\
\frac{Y_{Rij} Y_{R} + P_R^{i,j} P_R^{j,i}}{Y_{R \ominus i} Y_{R \ominus j}}. & \mbox{for $i \in S - S^{\max}$, $j =i^+$, and $R = S_{j}-\{i,j\}$}.
\end{cases}
\end{equation}

As is discussed in \cite{LP}, the exchange Laurent polynomials $\hF_i$ uniquely determine the exchange polynomials $F_i$.  Note that the expression for $\hF_i$ given by \eqref{E:hFi} is not in terms of cluster variables of $t_\s$, but in terms of variables in many different seeds.  In the course of our proof of Theorem \ref{thm:main}, it will be established that $\hF_i^\s$ do give legitimate exchange Laurent polynomials when written in terms of the cluster variables of $t_\s$.

Suppose $\s = \emptyset$, and thus $S = \emptyset$.  Then we have 
$$
P^{i,j}_\emptyset = \begin{cases} 1 &\mbox{there is an edge $i \to j$ in $\Gamma$ or $i = j$}\\
0 & \mbox{otherwise}.
\end{cases}
$$
Thus the seed $t_\emptyset$ consists of the cluster variables $X_i$, and the linear exchange polynomials (which are also the exchange Laurent polynomials $\hF_i$)
$$
F_i = A_i + \sum_{i \to j} X_j.
$$
We call $t_\emptyset$ the {\it initial seed}.  Our main theorem is the following:

\begin{theorem}\label{thm:main}
There is a normalized LP algebra $\A = \A_\Gamma$ with seeds exactly the $\{t_\s \mid \s \in \C\}$, and mutations given by $\mu_i(t_\s) = t_{\mu_i(\s)}$.
\end{theorem}

Thus $\A_\Gamma$ has
\begin{enumerate}
\item
clusters  $t_\s$ labeled by maximal nested collections $\s \in \C$.
\item
cluster variables $X_i$ and $Y_S$ for strongly connected components $S$ of $\Gamma$.
\item
exchange graph equal to the exchange graph on maximal nested collections $\C$ of Section \ref{sec:exc}.
\end{enumerate}

The proof of Theorem \ref{thm:main} will occupy the remainder of this section.  We now outline the strategy of the proof.  In Section \ref{ss:activation} we will describe the seeds $t_{\vec{S}}$ that can be obtained from the initial seed $t_\emptyset$ by mutating some sequence $\vec{S} =(s_1,s_2,\ldots,s_r) \subset [n]$, one at a time, without mutating any $i \in [n]$ twice.  These correspond to the activations of nested collections discussed in Section \ref{sec:nestmutations}.  We first describe only the exchange Laurent polynomials of these seeds for external (not-yet activated) vertices which is enough to carry through the inductive calculation, and show that $t_{\vec{S}}$ depends only on $\s$ and not the order of activation $\vec{S}$.  In Section \ref{ss:internal}, we then compute the exchange Laurent polynomials for internal vertices, thus establishing that $t_\s = t_{\vec{S}}$.  Finally, in Section \ref{ss:proof}, we shall show that $\mu_i(t_\s) = t_{\mu_i(\s)}$ for activated $i \in S - S^{\max}$.

%\begin{theorem}
%The LP algebra generated by this initial seed has the following properties:
%\begin{itemize}
% \item its cluster variables are the $X_i$ and the $Y_I$, for $I \in \I$;
% \item its clusters are of the following form: for a maximal nested collection $\s$ with support $S \subset [n]$, a cluster consists of $X_i$, $i \not \in S$ and $Y_{S_i}$, $S_i \in \s$;
% \item the exchanges are between the clusters that differ in exactly one variable.
%\end{itemize}
%\end{theorem}
%Note that the last condition uniquely defines the exchange Laurent polynomials $\hF$, since variables in each cluster generate the field of rational functions in $X_i$. In turn, the $\hF$-s uniquely define the exchange polynomials $F$. 

\begin{example}
For the graph in Example \ref{ex:first}, take $\s = \{ \{1\},\{1,2\},\{1,2,3\}\}$.  We have $t_\s = \mu_3(\mu_2(\mu_1(t_\emptyset)))$.  Computing these mutations (and making choices of units so that all exchange polynomials are positive) we obtain that the seed $t_\s$ is given by
\begin{multline*}
(Y_1, F_1= 1+Y_{12})\\
\shoveleft{(Y_{12}, F_2 =1+Y_1^2+Y_1(2+Y_{123}))}\\ \shoveleft{(Y_{123}, F_3 = X_4(1+Y_1)(1+Y_{12})+A_1(1+Y_1+Y_{12})+A_2 Y_1(1+Y_1)+A_3 Y_1 Y_{12})} \\ 
\shoveleft{(X_4, F_4 = A_1(1+Y_{12}+Y_1^2+Y_1(2+Y_{123})+Y_1Y_{12})+} \\ A_2(Y_1^3+Y_1^2(2+Y_{123})+Y_1)+A_3(Y_1^2Y_{12}+Y_1Y_{12})+A_4Y_1Y_{12}Y_{123}).\end{multline*} 
Here, all exchange polynomials have been written in terms of the cluster variables $\{Y_1,Y_{12},Y_{123},X_4\}$ of $t_\s$.  One checks that $\hF_3 = F_3/Y_{1}$, and $\hF_4 = F_4/(Y_1Y_{12})$, and thus:
\begin{align*}
\hF_3 &= \frac{X_4(1+Y_1)(1+Y_{12})+A_1(1+Y_1+Y_{12})+A_2 Y_1(1+Y_1)+A_3 Y_1 Y_{12}}{Y_1} \\
&=X_4(Y_{12}+Y_2 +1)+A_1(1+Y_2)+A_2(1+Y_1)+A_3Y_{12}
\end{align*}
and 
\begin{align*}
\hF_4 &= \frac{1}{Y_1Y_{12}} (A_1(1+Y_{12}+Y_1^2+Y_1(2+Y_{123})+Y_1Y_{12}) \\
&+A_2(Y_1^3+Y_1^2(2+Y_{123})+Y_1)+A_3(Y_1^2Y_{12}+Y_1Y_{12})+A_4Y_1Y_{12}Y_{123})\\
&=A_1(Y_3+1)+A_2 Y_{13}+A_3(Y_1+1)+A_4Y_{123}
\end{align*}
agreeing with \eqref{E:hFi}.  For example, the coefficient $Y_{12}+Y_2 +1$ of $X_4$ in $\hF_3$ is a summation over the following three paths from $3$ to $4$ via $S = \{1,2\}$: 
$$3 \longrightarrow 4, \;\;\; 3 \longrightarrow 1 \longrightarrow 4, \;\;\; 3 \longrightarrow 2 \longrightarrow 1 \longrightarrow 4.$$
\end{example}

\subsection{Rigidity and irreducibility of exchange Laurent polynomials}
In this subsection we will make some general observations concerning the polynomials defined in \eqref{E:hFi}.  Let $\L(t)$ denote the Laurent polynomial ring of the seed $t$.  If $t$ has cluster variables $x_1,x_2,\ldots,x_n$, then $\L(t) = R[x_1^{\pm 1},\ldots,x_n^{\pm 1}]$.

\begin{lemma}
The cluster variables $X_i^\s$ of $t_\s$ are algebraically independent.
\end{lemma}
\begin{proof}
Follows easily from Proposition \ref{prop:Yrat}.
\end{proof}

We note that all the denominators of $\hF^\s_i$ defined in \eqref{E:hFi} are monomials in $\L(t_\s)$.  For $i \notin S$ or $i \in S^{\max}$, we have that $S \ominus i$ is the union of some of the strongly connected components of $S$.  For $i \in S-S^{\max}$,  both $S \ominus i$ and $S\ominus j$ are unions of some of the strongly connected components of $R$.  The strongly connected components of $R$ are exactly the $S_k \in \s$ such that $k^+ = i$ or $k^+ = j$.

\begin{prop}\label{prop:irred}
Let $\s \in \C$ be a maximal nested collection, and suppose it is known that
\begin{equation}\label{E:laurent}
\{Y_I \mid I \subset S\} \subset \L(t_\s).
\end{equation}
Then the exchange polynomials $\hF_i^\s$ of \eqref{E:hFi} are irreducible in $\L(t_\s)$.
\end{prop}
\begin{proof}
Suppose first $i \in S^{\max}$ or $i \notin S$.  The denominator $Y_{S \ominus i}$ is a monomial in the cluster variables of $\L(t_\s)$, and each $P^{i,j}_S$ is a polynomial in the $\{Y_I \mid I \subset S\}$, so it follows from \eqref{E:laurent} that we have $\hF_i^\s \propto MA_i + N$ for a monomial $M$ in the $\{X_k^\s \mid k \in [n]\}$ and a polynomials $N \in R[X_i^\s]$ which does not involve $A_i$.  It follows easily that $\hF_i^\s$ is irreducible in $\L(t_\s)$.

Now if $i \in S - S^{\max}$, then again the denominator $Y_{R\ominus i}Y_{R \ominus j}$ is checked to be a unit in $\L(t_\s)$.  Now we apply the same argument as before but with the cluster variable $Y_{Rij}$ instead of $A_i$.  Note that by \eqref{E:laurent} and Proposition \ref{prop:Yrat}, $Y_I$ for $I \subset R$ will be a Laurent polynomial not involving $Y_{Rij}$.
\end{proof}

Equation \eqref{E:laurent} will follow from Corollary \ref{C:laurent} below.

It turns out one does not have to compute the exchange Laurent polynomials explicitly, but only up to a monomial factor in the cluster variables of the corresponding cluster.  This will simplify the subsequent calculations significantly. Let $\A(t_\emptyset)$ denote the LP algebra generated by $t_\emptyset$.  The LP algebra $\A(t_\emptyset)$ is only defined up to choices of a unit $\pm 1 \in R$ for all the mutations; we fix these choices by insisting that all cluster variables of $\A(t_\emptyset)$ are Laurent polynomials in $X_1,X_2,\ldots,X_n$ with positive coefficients.  That such a choice can be made will follow from our constructions.  It is also clear that this choice is compatible with Theorem \ref{thm:comm}.

\begin{prop}\label{prop:rigid}
Let $\s \in \C$, and assume it is known that $\{Y_I \mid I \subset S\}$ are all cluster variables of $A(t_\emptyset)$.  Suppose $t = \{(X_i^\s,F_i)\}$ is a seed of $A(t_\emptyset)$ with the same cluster variables $X_i^\s$ as $t_\s$ but possibly different exchange polynomials.  Suppose also that the exchange Laurent polynomials $\hF_i$ for $i \notin S$ or $i \in S^{\max}$ are all known to be equal to $\hF_i^\s$ of \eqref{E:hFi} up to a monomial in the $\{X_r^\s \mid r \in [n]\}$.  Then $\hF_i = \hF_i^\s$ for all $i \notin S$ or $i \in S^{\max}$.  Furthermore if for any $i \in S-S^{\max}$ we have that $\hF_i$ is equal to $\hF^\s_i$ up to a monomial factor, then $\hF_i = \hF^\s_i$ as well.
\end{prop}
\begin{proof}
First suppose that $i \notin S$ or $i \in S^{\max}$.  We need to argue that $\hF_i$ is equal to
$$\hF^\s_i = \frac{\sum_{j \not \in Si} P_S^{i,j} X_j + \sum_{j \in Si} P_S^{i,j} A_j}{Y_{S \ominus i}}.$$
Applying Lemma \ref{lem:Fpoly} and Lemma \ref{lem:monden}, we conclude $\hF_i$ is obtained from $\hF^\s_i$ by possibly dividing by {\it {more}} of the $Y$ variables in $t$.  The $X$-variables cannot occur in $\hF_i/F_i$ since for $k \notin S$, $\hF_k$ depends on $A_k$ which does not occur in $\hF_i$.

%Here we use the induction assumption that all $Y$-s supported on $S \cup k$ are cluster variables. This is exactly the right size since at the moment we have $S$ of certain size, we know that are the $Y$-s of size {\it {one larger}}, since we know the current $\hF$ for the external variables, and we know Lemma \ref{lem:maxmut}. 

On the other hand, if we have $\hF_i = \hF_i^\s M$ where $M$ is a (negatively-powered) monomial in the $Y$ variables, then the cluster variable $X'_i = \hF_i/X_i$ in $\mu_i(t)$ would be equal to 
$$
\frac{\hF_i}{X_i} = M\frac{\hF_i^\s}{X_i} = MY_{S \oplus i},$$ where the last equality is by Lemma \ref{lem:maxmut}. This is impossible by Lemma \ref{lem:irr} and the fact that the resulting cluster variable has to be a Laurent polynomial in the $t_\emptyset$ cluster. 

Now suppose $i \in S-S^{\max}$, and set $j = i^+$ and $R = S_j -\{i,j\}$.  Then
$$\hF^{\s}_i = \frac{Y_{Rij} Y_{R} + P_R^{i,j} P_R^{j,i}}{Y_{R \ominus i} Y_{R \ominus j}}.$$
Just as before, none of the $X_k$ variables can be in the denominator $\hF_i/F_i$.  The rest of the argument is the same, with Lemma \ref{lem:Yex} and $Y_{R \oplus j}$ instead of Lemma \ref{lem:maxmut} and $Y_{S \oplus i}$.
\end{proof}

\subsection{Clusters reachable from the initial seed via an activation sequence}
\label{ss:activation}

%Call a cluster {\it {special}} if it is obtained by the following {\it {activation}} procedure:
%\begin{itemize}
% \item start with the intitial cluster $\{X_i\}_{i \in [n]}$;
% \item mutate variables in some order, never mutating a variable more than once.
%\end{itemize}
%
%We call the $i$-s such that the $X_i$ has already been mutated {\it {active}}, and we call the rest of them {\it {external}}. Clearly in any special cluster 
%there are initial $X$ variables that have not been activated yet, and there are some other variables obtained by mutation.
%Assume we are in the special cluster obtained by activating the set $S = \{i_1, \ldots, i_k\}$ in that order. To emphasize that we are dealing with an ordered set, we shall use notation $t_{\vec{S}}$ for this cluster.
Let $\vec{S} = (s_1,s_2,\ldots,s_r)$ be an activation sequence giving a maximal nested collection $\s$ with support $S$.  Denote by $t_{\vec{S}}$ the seed $\mu_{s_r} \circ \cdots \circ \mu_{s_1}(t_\emptyset)$ of $\A(t_\emptyset)$. 

\begin{prop} \label{prop:spec}
In $t_{\vec{S}}$
 \begin{enumerate}
\item the cluster variables are those of $t_\s$ as defined in \eqref{E:cluster}
 \item for all $i \notin S$ or $i \in S^{\max}$ the exchange Laurent polynomials $\hF_i$ agree with the definition \eqref{E:hFi} of $\hF_i^\s$ for $t_\s$ .
 \item the seed $t_{\vec{S}}$ depends only on $\s$, and not on the actual activation sequence $\vec{S}$.
\end{enumerate}
\end{prop}

\begin{corollary}\label{C:laurent}
For each strongly connected subset $T \subset [n]$ and each $\s \in \C$ we have $Y_T \in \L(t_\s)$.
\end{corollary}
\begin{proof}
Follows immediately from Proposition \ref{prop:spec}(1) and Theorem \ref{thm:LP}.
\end{proof}

\begin{proof}[Proof of Proposition \ref{prop:spec}]
We prove the three claims by simultaneous induction on the size of the support $S$. The base case $|S| = 0$ is clear.

Assume the result has been shown for $\vec{S}'=(s_1,s_2,\ldots,s_{r-1})$ and we now want to verify it for $\vec{S} = (s_1,s_2,\ldots,s_{r-1},k)$ where $k\notin S'$.  Write $\s'$ for the maximal nested collection with activation sequence $\vec{S}'$, and similarly for $\s$.  Note that $S = S' \cup \{k\}$.  We shall write $\hF_i$ (resp. $\hF'_i$) for the Laurent exchange polynomials of the seed $t = t_{\vec{S}}$ (resp. $t' = t_{\vec{S'}}$).
By definition and the inductive hypothesis, the cluster variable $X_k$ in $t_{\vec{S}'}$ is exchanged for 
$$
\hF'_i/X_k = \frac{\sum_{j \not \in Sk} P_S^{k,j} X_j + \sum_{j \in Sk} P_S^{k,j} A_j}{Y_{S \ominus k}X_k} = Y_{S \oplus k}
$$
by Lemma \ref{lem:maxmut}.  But $\s = \s' \cup \{S \oplus k\}$, so we have verified claim (1).

We now consider claim (2).  We first make the general observation that by Theorem \ref{thm:LP}, each $\hF'_i$ of \eqref{E:hFi} can be seen to lie in $\L(t')$, since each $Y_T$ for $T \subset S'$ is known to be a cluster variable of $\A(t_\emptyset)$.

We know that $\hF_k = \hF'_k$ and this easily agrees with \eqref{E:hFi}.  

Now suppose $i \in (S')^{\max} \cap S^{\max}$ is maximal in both $S'$ and in $S$.  Thus $i$ and $k$ are not in the same strongly connected component of $S$, so either there is no path from $i$ to $k$ through $S$, or there is no path from $k$ to $i$ through $S$.

In the first case, we see that the cluster variable $X_k$ does not appear in $\hF'_i$.   Therefore $F_i = F'_i$, so $\hF_i = \hF'_i$ up to a monomial in $\L(t)$.  But we note that $\hF'_i = \hF^{\s'}_i$ is actually equal to the predicted $\hF^\s_i$ for the following reason. Since we cannot get from $i$ to $S \oplus k$ via $S$ (resp. $S'$), each of the $P_S^{i,\ell}$ (resp. $P_{S'}^{i_\ell}$) actually factors, producing $Y_{S \oplus k}$ (resp. $Y_{(S \oplus k)-\{k\}}$) factors. This same factor appears in the denominator of $\hF^{\s}_i$ (resp. $\hF^{\s'}_i$) which is  $Y_{S \ominus i}$ (resp. $Y_{S' \ominus i}$).  After these factors are canceled, we see that $\hF^\s_i = \hF^{\s'}_i$.   It follows from Proposition \ref{prop:rigid} that $\hF_i = \hF^\s_i$.

In the second case, consider the seed $t_{\vec{S''}}$ with $S'' = S'-\{i\}$ and such that activation of $i$ gives $t'$.  This seed exists by the inductive hypothesis (3).  We have two cluster variables in $t_{\vec{S''}}$, $X_i$ and $X_k$, such that $X_i$ does not appear in $\hF_k$ (and thus also not in $F_k$). Then by Theorem \ref{thm:comm} we know the two mutations commute, and $\hF_i$ after activating $i$ and then $k$ is the same as the exchange Laurent polynomial $\hF'''_i$ in the seed $t''' = \mu_k(t_{\vec{S''}})$ after just activating $k$.  But $t'''$ falls inside our inductive hypothesis.  It remains to note that $P_{S-\{i\}}^{i, \ell} = P_{S}^{i,\ell}$, and thus $\hF_i = \hF'''_i = \hF^{\s'''}_i = \hF^\s_i$, establishing (2) in this case.

Consider now the case when $i \notin S$ is external. The case where there is no path from $i$ to $k$ through $S$ is treated the same way as for $i \in (S')^{\max} \cap S^{\max}$. Assume a path from $i$ to $k$ through $S$ exists.  By the induction assumption, we have 
$$\hF'_i = \frac{\sum_{\ell \not \in Si} P_S^{i,\ell} X_{\ell} + \sum_{\ell \in S i} P_S^{i,\ell} A_{\ell}}{Y_{S \ominus i}}$$
and 
$$\hF'_k|_{X_i=0} = \frac{\sum_{\ell \not \in Sik} P_S^{k,\ell} X_{\ell} + \sum_{\ell \in Sk} P_S^{k,\ell} A_{\ell}}{Y_{S \ominus k}}.$$
By Remark \ref{rem:subs}, we can calculate $\hF_i$ by substituting into $\hF'_i$.
%
%Clearly, $\hF'_i$ is linear in $X_k$, while $\hF'_k$ is either constant or linear in $X_i$. Then the same must be true for $F'_i$ and $F'_k$. Indeed, there are no $X_j$'s in the denominator of any $\hF'_j$.
%
%Thus exchange polynomials of $X_i$ and $X_k$ are linear in each other, both before and after adding denominators, and are not involved in denominators of each other. This implies that we can do the substitution $$X_k \longleftarrow \frac{\hF_k}{Y_{S \oplus k}}$$ in $\hF_i$, instead of doing $$X_k \longleftarrow \frac{F_k}{Y_{S \oplus k}}$$ in $F_i$, with the result being different only by a monomial factor in other $Y$ variables in this cluster. Thus, the argument breaks into two steps: we shall first find the result of the first substitution, thus finding the resulting $\hF_i$ up to a monomial factor. Then we shall argue that the denominator is exactly right to give $\hF_i$ predicted by the theorem.

Substituting $\hF'_k|_{X_i=0}/Y_{S\oplus k}$ for $X_k$ in $\hF'_i$ gives
$$
\frac{1}{Y_{S \oplus k}} \sum_{\ell \neq i}  \left( \frac{P_S^{k,\ell}}{Y_{S \ominus k}} \cdot \frac{P_S^{i,k}}{Y_{S \ominus i}} + \frac{P_S^{i,\ell}}{Y_{S \ominus i}} \cdot Y_{S \oplus k} \right)\X_\ell + \frac{P_S^{i,i}}{Y_{S \ominus i}} \cdot Y_{S \oplus k} X_i  = \frac{1}{Y_{S \oplus k}} \sum_\ell P_{Sk}^{i,\ell} \frac{Y_S}{Y_{S \ominus k} Y_{S \ominus i}}\X_\ell
$$
where $\X_\ell$ is equal to $X_\ell$ or $A_\ell$ and we have applied Lemma \ref{lem:ikj}, noting that $Y_{S \ominus k} Y_{S \oplus k} = Y_{Sk}$.  The factor $ \frac{Y_S}{Y_{S \ominus k} Y_{S \ominus i}}$ does not depend on $j$ so together with Proposition \ref{prop:irred} we have that 
$$\hF_i \propto \sum_{j \not \in Sik} P_{Sk}^{i,j} X_{j} + \sum_{j \in Sik} P_{Sk}^{i,j} A_{j}.$$  The proof of (2) is then completed by applying Proposition \ref{prop:rigid}.

Finally, to establish (3), let $\vec{S}''=(s_1,s_2,\ldots,s_{r-2})$ and suppose $s_{r-1}$ and $s_{r-2}$ are exchangeable in $\vec{S} = (s_1,s_2,\ldots,s_{r-2},s_{r-1},s_r)$.  Then $s_{r-1}$ and $s_r$ are not strongly connected via $S''$, so it follows that one of $\hF_{s_{r-1}}$ and $\hF_{s_r}$ does not involve the other's cluster variable ($X_{s_r-1}$ or $X_{s_r}$).
By Theorem \ref{thm:comm} we have $\mu_{s_r} \mu_{s_{r-1}}(t_{\vec{S}''}) = \mu_{s_{r-1}} \mu_{s_{r}}(t_{\vec{S}''})$.  This together with the inductive hypothesis completes the proof of (3).
\end{proof}

\subsection{The internal exchange polynomials}
\label{ss:internal}

\begin{prop} \label{prop:Ymut}
The exchange Laurent polynomials $\hF_i$ of $t_{\vec{S}}$ for $i \in S - S^{\max}$ agree with those of $t_\s$ defined in \eqref{E:hFi}.
\end{prop}

Together with Proposition \ref{prop:spec} we thus have $t_{\vec{S}} = t_\s$.

\begin{proof}
Fix $i \in S - S^{\max}$ and set $j = i^+$ and $R = S_j -\{i,j\}$.  Note that $S_i = R \oplus i$ and $S_j = Rij = Ri \oplus j$. 

It suffices to show that the statement is true when $j$ is the unique maximal element of $\s$, so that $S_j = S$.  To see this, observe by Lemma \ref{lem:exc} and Proposition \ref{prop:spec} that $t_{\vec{S}}$ can be created by first activating all the elements in $S_j$ (in some order), and then activating elements outside of $S_j$ in some order.  But by Proposition \ref{prop:Yrat} the expression $\frac{Y_{Rij} Y_{R} + P_R^{i,j} P_R^{j,i}}{Y_{R \ominus i} Y_{R \ominus j}}$ can be expressed purely in terms of the $Y$ variables inside $S_j$, and does not depend on $X$ variables. Then Theorem \ref{thm:comm} states that activation of variables outside $S_j$ commutes with mutation of $Y_{S_i}$, and in particular activation of variables outside $S_j$ preserves the value of $\hF_i$.  We shall thus assume from now on that $S = S_j$.

%Thus $Y_{S_i}$ must mutate into the same variable (in fact, according to Lemma \ref{lem:Yex}, into $Y_{R \oplus j}$) both before and after the activations. Thus, the $\hF_i$ is the same at any moment as it was at the moment after $j$ was activated.  

Let $t' = t_{\vec{S}'}$ be the cluster with $i$  maximal such that mutating at $j$ gives $t  = t_{\vec{S}}$, so we have $S' \cup \{j\} = S = S_j$.  Denote by $F'_i$ the exchange polynomials of the cluster $t'$.  By Proposition \ref{prop:spec},
$$\hF'_i = \frac{\sum_{\ell \not \in Ri} P_R^{i,\ell} X_{\ell} + \sum_{\ell \in Ri} P_R^{i,\ell} A_{\ell}}{Y_{R \ominus i}}$$
and 
$$\hF'_j = \frac{\sum_{\ell \not \in Rij} P_{Ri}^{j,\ell} X_{\ell} + \sum_{\ell \in Rij} P_{Ri}^{j,\ell} A_{\ell}}{Y_{Ri \ominus j}}.$$
Use the notation 
\begin{equation}\label{E:Xtilde}
\X_{\ell} = 
\begin{cases}
X_{\ell} & \text{if $\ell \not \in R \cup i \cup j$;}\\
A_{\ell} & \text{otherwise.}
\end{cases}
\end{equation}
Then noting that $Ri \ominus j = \emptyset$, we have
$$\hF'_j = \frac{\sum_{\ell} P_{Ri}^{j,\ell} \X_{\ell}}{Y_{(Ri) \ominus j}} = \sum_{\ell} P_{Ri}^{j,\ell} \X_{\ell}.$$
We can use Lemma \ref{lem:ikj} to obtain
$$\hF'_j = \frac{\sum_{\ell} P_{R}^{j,\ell} Y_{R i} \X_{\ell} + \sum_{\ell \neq j} P_{R}^{j,i} P_{R}^{i,\ell} \X_{\ell}}{Y_{R}}.$$
Applying Proposition \ref{prop:Yrat} to this expression, we can write it as a rational function in the cluster variables of $t_{\s'}$.  The variable $Y_{S_i} = Y_{Ri}$ does not appear in $P^{j,\ell}_R$,  $P_{R}^{j,i}$, or $P_{R}^{i,\ell}$ so we see that setting $Y_{S_i} = 0$ in $\hF'_j$ gives 
$$\frac{\sum_{\ell \neq j} P_{R}^{j,i} P_{R}^{i,\ell} \X_{\ell}}{Y_{R}}.$$
Substituting for $X_j$ in $\hF'_i$ gives $$(\hF'_i)_{X_j \longleftarrow \frac{\sum_{\ell} P_{R}^{j,i} P_{R}^{i,\ell} \X_{\ell}}{Y_{R} Y_{S_j}}} = \frac{1}{Y_{R \ominus i}} \left( P_R^{i,j} \frac{\sum_{\ell} P_{R}^{j,i} P_{R}^{i,\ell} \X_{\ell}}{Y_{R} Y_{S_j}} + \sum_{\ell} P_{R}^{i,\ell} \X_{\ell}\right),$$ where throughout the summation is over $\ell \not = j$.
Dividing by the common factor $\sum_{\ell} P_{R}^{i,\ell} \X_{\ell}$ with $\hF'_j|_{Y_{S_i}=0}$ and using Proposition \ref{prop:irred}, we obtain
$$\hF_i \propto Y_{R ij} Y_{R} + P_R^{i,j} P_R^{j,i}.$$
The claim then follows from Proposition \ref{prop:rigid}.

%up to a monomial in the other $Y$-s (any other factor to be killed must divide $Y_{R \cup i \cup j} Y_{R}$, and thus also has to be a monomial).
\end{proof}

%\begin{lemma}\label{lem:irred}
%$Y_{Rij} Y_{R} + P_R^{i,j} P_R^{j,i}$ is a Laurent polynomial in the cluster variables of $t_\s$, and is irreducible in $\L(t_\s)$.
%\end{lemma}
%\begin{proof}
%The Laurentness statement follows from Theorems \ref{thm:LP} and \ref{thm:spec}.  Thus $Y_{R ij} Y_{R} + P_R^{i,j} P_R^{j,i}$ is up to a unit in $\L(t_\s)$ equal to a polynomial $f = Y_{Rij}M + P$ where $M$ is a monomial and $P$ is a polynomial.  It follows from Proposition \ref{prop:Yrat} that one can ensure that neither $M$ nor $P$ involves the variable $Y_{Rij}$.  But since $f$ is linear in $Y_{Rij}$ this means that any factor of $f$ must be a divisor of $M$, and thus a monomial.
%\end{proof}

\subsection{Proof of Theorem \ref{thm:main}}
\label{ss:proof}
Let $t_\s$ be one of the clusters in $\A(t_\emptyset)$ that have been studied in Propositions \ref{prop:spec} and \ref{prop:Ymut}, and $t' = \mu_i(t_\s)$ be the cluster obtained by mutation at $i$, where $i \in S - S^{\max}$.  Our aim is to show that $t' = t_{\mu_i(\s)}$.  Set $\s' = \mu_i(\s)$.

Proposition \ref{prop:Ymut} and Lemma \ref{lem:Yex} imply that the cluster variables of $t'$ and $t_{\s'}$ are the same. It remains to check that the exchange polynomials $F'_k$ of $t'$ are also correct, and for that it suffices to check that the $\hF'_k$'s are correct.  Just as in the proof of Proposition \ref{prop:Ymut}, we may assume that $S = S_{i^+}$, since all further activating mutations commute with the one exchanging $Y_{S_i}$ for $Y_{(S - \{i\}) \oplus i^+}$. Let $j = i^+$ as before, and let $K \subset S$ be the set of $k$ such that $k^+ = i$ or $k^+ = j$. Since $\hF_i$ does not depend on external $X$ variables, its mutation commutes with theirs by Theorem \ref{thm:comm}, and thus does not change the corresponding $\hF$. The same is true for elements of $r \in S$ different from $j$ and not in $K$: their exchange polynomial $F_r$ does not depend on $Y_{S_i}$ and also is not equal to $F_i$ (even up to a unit in $\L(t_\s)$). 

Thus we only need to check that $\hF_j$ and $\hF_k$ for $k \in K$ mutate correctly. 
Set $R = S_j - \{i,j\}$ as before. Setting $Y_{Rij}=0$ inside 
$$\hF_i = \frac{Y_{Rij} Y_{R} + P_R^{i,j} P_R^{j,i}}{Y_{R \ominus i} Y_{R \ominus j}}$$
we obtain
$$\frac{P_R^{i,j} P_R^{j,i}}{Y_{R \ominus i} Y_{R \ominus j}},$$ since by Proposition \ref{prop:Yrat} this expression does not depend on $Y_{Rij}$.
Plugging 
$$Y_{R \oplus i} \longleftarrow \frac{P_R^{i,j} P_R^{j,i}}{Y_{R \ominus i} Y_{R \ominus j} Y_{R \oplus j}}$$
in 
$$\hF_j = \sum_{\ell} P_{Ri}^{j,\ell} \X_{\ell} = \frac{\sum_{\ell} P_{R}^{j,\ell} Y_{Ri} \X_{\ell} + \sum_{\ell \neq j} P_{R}^{j,i} P_{R}^{i,\ell} \X_{\ell}}{Y_{R}}$$
where we are using the notation of \eqref{E:Xtilde}, we get 
$$\frac{\sum_{\ell} P_{R}^{j,\ell} \frac{P_R^{i,j} P_R^{j,i}}{Y_{R \ominus i} Y_{R j}} Y_{R \ominus i} \X_{\ell} + \sum_{\ell} P_{R}^{j,i} P_{R}^{i,\ell} \X_{\ell}}{Y_{R}} =
\frac{\frac{P_R^{i,j} P_R^{j,i}}{Y_{Rj}} \sum_{\ell} P_{R}^{j,\ell}  \X_{\ell} + P_{R}^{j,i} \sum_{\ell} P_{R}^{i,\ell} \X_{\ell}}{Y_{R}}.$$
Canceling out the common factor $P_{R}^{j,i}$ we get
$$ \hF'_j \propto \frac{\frac{P_R^{i,j}}{Y_{Rj}} \sum_{\ell} P_{R}^{j,\ell}  \X_{\ell} + \sum_{\ell} P_{R}^{i,\ell} \X_{\ell}}{Y_{R}}=
\frac{\sum_{\ell} P_{Rj}^{i,\ell}  \X_{\ell}}{Y_{R j}}$$
by Lemma \ref{lem:ikj} (for the last equality) and Proposition \ref{prop:irred}.  We then see that $\hF'_j = \hF^{\s'}_j$ by Proposition \ref{prop:rigid}.

It remains to check what happens with $\hF_k$, for $k \in K$.  Let $T = S - i -k -j = R - k$.   We have
\begin{equation}\label{E:Fi}\hF_i = \frac{Y_{T ijk} Y_{Tk} + P_{Tk}^{i,j} P_{Tk}^{j,i}}{Y_{Tk \ominus i} Y_{Tk \ominus j}}.
\end{equation}

%Then there are cases to consider, depending on which of the $X_i$, $X_j$ and $X_k$ appears in the $F$ of which at the moment when $T$ is activated. Since we know $S = S_j$, we are assuming that the resulting graph on $i$, $j$ and $k$ is strongly connected.

%Take the activation sequence which produces the current set $S$ with $k$, $i$ and $j$ activated last, in that order. 

\noindent {\bf Case $k^+=j$:}
Suppose first that $k^+ = j$, so that $k$ and $i$ are not strongly connected via $T$, and thus $P_T^{i,k}P_T^{k,i} = 0$.  We have
\begin{equation}\label{E:Fk} \hF_k= \frac{Y_{Tikj} Y_{Ti} + P_{Ti}^{k,j} P_{Ti}^{j,k}}{Y_{Ti \ominus k} Y_{Ti \ominus j}}
\end{equation}
Applying Lemma \ref{lem:ikj} to $P_{Ti}^{k,j}$ and $P_{Ti}^{j,k}$ gives
$$
P_{Ti}^{k,j}Y_T = P^{k,i}_TP^{i,j}_T + P^{k,j}_T Y_{Ti} \qquad P_{Ti}^{j,k}Y_T = P^{j,i}_TP^{i,k}_T + P^{j,k}_T Y_{Ti}.
$$
Using $Y_{Ti} = Y_{T \ominus i}Y_{T \oplus i}$ and $Y_{Ti \ominus k} = Y_{(T\ominus i) \ominus k} Y_{T \oplus i}$ we get
\begin{equation}\label{E:Fk2}
\hF_k = \frac{Y_T^2Y_{S}Y_{T \ominus i} + Y_{T \ominus i} (P^{k,i}_TP^{i,j}_TP^{j,k}_T +P^{j,i}_TP^{i,k}_TP^{k,j}_T) + P^{k,j}_TP^{j,k}_T Y_{T \ominus i}Y_{Ti}}{Y_{(T\ominus i) \ominus k}Y_{Ti \ominus j}Y_T^2}
\end{equation}

\noindent {\bf Subcase $P^{k,j}_TP^{j,k}_T = 0$:} Thus $k$ and $j$ are not strongly connected via $T$ and in $\s'$, we have $k^+ = i$.  In this subcase, $Y_{S_i}$ only occurs in the denominator of $\hF_k$, so this means $F_k$ does not involve $Y_{S_i}$.  It follows that $F'_k = F_k$.  If $t' = t_{\s'}$, according to Theorem \ref{thm:Ymut} we should have
\begin{equation}\label{E:Fprimek}
\hF'_k = \frac{Y_T^2Y_{Tikj} Y_{Tj} + P_{Tj}^{k,i} P_{Tj}^{i,k}}{Y_{Tj \ominus k} Y_{Tj \ominus i}} = \frac{Y_{S}Y_{T \ominus j} + Y_{T \ominus j} (P^{k,i}_TP^{i,j}_TP^{j,k}_T +P^{j,i}_TP^{i,k}_TP^{k,j}_T)}{Y_T^2Y_{(T\ominus j) \ominus k}Y_{Tj \ominus i}}
\end{equation}

 If $i$ and $j$ are strongly connected via $T$, then $Tj \ominus i = Ti \ominus j$, and we see that \eqref{E:Fprimek} and \eqref{E:Fk2} differ (multiplicatively) by factors of $Y_V$ for variables $V$ which are present in both $t_\s$ and $t'$.  This implies that up to a monomial factor \eqref{E:Fk2} is the correct formula for $\hF'_k$ in $t'$, and we may now apply Proposition \ref{prop:rigid}.

If $i$ and $j$ are not strongly connected via $T$.  Then $Y_{Tj \ominus i}$ has a factor of the new cluster variable $Y_{T \oplus j}$ in $t'$ while $Y_{Ti \ominus j}$ has a factor of the old cluster variable $Y_{T \oplus i}$ in $t_\s$.  These factors are of course not present in the corresponding non-Laurent exchange polynomials $F_k$.  In any case, we still conclude that \eqref{E:Fk2} is the correct formula for $\hF'_k$ in $t'$.

\smallskip
\noindent
{\bf Subcase $P^{k,j}_TP^{j,k}_T \neq 0$:}
Thus $j$ and $k$ are strongly connected via $T$, and in $\s'$, we have $k^+ = j$.  In the denominator of \eqref{E:Fi}, we have $Y_{T k \ominus i} = Y_{T \oplus k} Y_{(Tk \ominus k) \ominus i}$ but $Y_{Tk \ominus j}$ does not involve the variable $Y_{T \oplus k}$.   Applying Lemma \ref{lem:ikj}, we get
$$
\hF_i = \frac{Y_T^2 Y_{Tijk}Y_{T\ominus k} + Y_{T\ominus k}(P^{i,k}_TP^{k,j}_TP^{j,i}_T + P^{i,j}_TP^{j,k}_TP^{k,i}_T) + P^{i,j}_TP^{j,i}_TY_{Tk}Y_{T\ominus k}}{Y_{(Tk \ominus k) \ominus i}Y_{Tk \ominus j}Y_T^2}
$$

None of the $P^{a,b}_T$ involve $Y_{T \oplus k}$ when expressed as an element of $\L(t_\s)$, so substituting $Y_{T \oplus k} = 0$, we obtain
$$
\hF_i|_{Y_{T \oplus k} =0}/Y_{Tk \oplus j}= \frac{ Y_T^2 Y_{Tijk}Y_{T\ominus k} + Y_{T\ominus k}(P^{i,k}_TP^{k,j}_TP^{j,i}_T + P^{i,j}_TP^{j,k}_TP^{k,i}_T)}{Y_{(Tk \ominus k) \ominus i}Y_{Tk \ominus j}Y_T^2Y_{Tk \oplus j}}
$$
Substituting this last expression for $Y_{T \oplus i} = Y_{S_i}$ in \eqref{E:Fk2} (again we use the fact that the expression in \eqref{E:Fk2} is linear in $Y_{S_i}$), we obtain, up to an overall monomial factor
\begin{align*}
& Y_T^4 Y_SY_{Tk} Y_{(Tk \ominus k) \ominus i} + ZY_T^2Y_{Tk} Y_{(Tk \ominus k) \ominus i} + P^{k,j}_TP^{j,k}_T Y_SY^3_T Y_{(Tk \ominus k)\ominus i} +P^{k,j}_TP^{j,k}_T Y_T Y_{(Tk \ominus k)\ominus i} Z \\
&= (Y_{Tkj}Y_T + P^{k,j}_TP^{j,k}_T)(Y_T^2Y_S+Z) Y_TY_{(Tk \ominus k)\ominus i}
\end{align*}
where $Z = P^{i,k}_TP^{k,j}_TP^{j,i}_T + P^{i,j}_TP^{j,k}_TP^{k,i}_T$, and we have used the equalities $Y_{T} Y_{(Tk \ominus k) \ominus i} = Y_{T \ominus i} Y_{T \ominus k}$ and $Y_{Tk} = Y_{Tk \ominus j} Y_{Tk \oplus j}$.  The factor $(Y_T^2Y_S+Z)$ is a common factor with $\hF_i|_{Y_{T \oplus k} =0}$, so we conclude from Proposition \ref{prop:irred} that 
$$
\hF'_k \propto Y_{Tkj}Y_T + P^{k,j}_TP^{j,k}_T.
$$
The claimed statement for $\hF'_k$ then follows from Proposition \ref{prop:rigid}.

\smallskip
\noindent {\bf Case $k^+=i$:}
In this case we have 
\begin{equation}\label{E:Fk3}
\hF_k = \frac{Y_{Tik} Y_{T} + P_{T}^{i,k} P_{T}^{k,i}}{Y_{T \ominus i} Y_{T \ominus k}}.
\end{equation}

\smallskip
\noindent {\bf Subcase $P^{j,k}_TP^{k,j}_T = 0$:}
In this subcase, we have $k^+ = i$ in $t_{\s'}$.

Applying Lemma \ref{lem:ikj} to \eqref{E:Fi}, and using $Y_{Tk \ominus j} = Y_{(Tk \ominus j) \ominus k} Y_{T \oplus k} $ we get
\begin{equation}\label{E:Fi2}
\hF_i = \frac{Y_T^2 Y_{Tijk}Y_{T\ominus k} + Y_{T\ominus k}(P^{i,k}_TP^{k,j}_TP^{j,i}_T + P^{i,j}_TP^{j,k}_TP^{k,i}_T) + P^{i,j}_TP^{j,i}_TY_{Tk}Y_{T\ominus k}}{Y_{(Tk \ominus j) \ominus k}Y_{Tk \ominus i}Y_T^2}
\end{equation}

Plugging $Y_{T \oplus k} = 0$ into \eqref{E:Fi2} we get
$$\hF_i|_{Y_{T \oplus k} =0} = \frac{Y_T^2Y_{Tijk} Y_{T \ominus k} + Y_{T \ominus k}Z}{Y_T^2Y_{Tk \ominus i} Y_{(T \ominus k) \ominus j}}$$
where $Z = P^{i,k}_TP^{k,j}_TP^{j,i}_T + P^{i,j}_TP^{j,k}_TP^{k,i}_T$.
Now substituting 
$$\frac{\hF_i|_{Y_{T \oplus k} =0}}{Y_{Tk \oplus j}} = \frac{Y_T^2Y_{Tijk} Y_{T \ominus k} + Y_{T \ominus k}Z}{Y_{Tk \oplus j}Y_T^2Y_{Tk \ominus i} Y_{(T \ominus k) \ominus j}}$$ for $Y_{S_i} = Y_{Tk \oplus i}$ in $\hF_k$, we get
$$
\frac{Y_T^2Y_{Tijk} Y_{T \ominus k} + Y_{T \ominus k}Z}{Y_T^2Y_{Tk \ominus i} Y_{(T \ominus k) \ominus j}Y_{Tk \oplus j}} Y_{Tk \ominus i} Y_{T} + P_{T}^{i,k} P_{T}^{k,i}
$$
up to a monomial factor in $\L(t')$.  Now $Y_{(T \ominus k) \ominus j}Y_{Tk \oplus j} = Y_{Tj \ominus k}$ and $Y_{Tj \ominus k} Y_T = Y_{Tj} Y_{T \ominus k}$ so we have
\begin{align*}
\hF'_k &\propto Y_T^2Y_{Tijk} Y_{T \ominus k} + Y_{T \ominus k}Z + Y_{Tj \ominus k}P_{T}^{i,k} P_{T}^{k,i} Y_T\\
& \propto Y_T^2Y_{Tijk} Y_{Tj} + Y_{Tj}Z + P_{T}^{i,k} P_{T}^{k,i} Y_{Tj}^2
\\& = Y_T^2 (Y_{Tijk}Y_{Tj} + P^{i,k}_{Tj}P^{k,i}_{Tj})
\end{align*}
where in the last step we have used Lemma \ref{lem:ikj}.  But this last expression, is up to a monomial in $\L(t')$ equal to the expression for $\hF'_k$ predicted by \eqref{E:hFi}.  The claimed statement for $\hF'_k$ then follows from Proposition \ref{prop:rigid}.

\smallskip
\noindent 
{\bf Subcase $P^{j,k}_TP^{k,j}_T \neq 0$}
In this case, we have $k^+ = j$ in $t_{\s'}$.  In \eqref{E:Fi} neither factor in the denominator is divisible by $Y_{T\oplus k}$, so substituting $Y_{T \oplus k} = 0$ into $\hF_i$, the term $Y_{Tijk} Y_{Tk}$ is killed since $Y_{Tk} = Y_{T \oplus k} Y_{T \ominus k}$.  For the other term $P_{Tk}^{i,j} P_{Tk}^{j,i}$ we apply Lemma \ref{lem:ikj} and obtain
$$\hF_i|_{Y_{T\oplus k} = 0} = \frac{P_{T}^{i,k} P_{T}^{k,j} P_{T}^{j,k} P_{T}^{k,i}}{Y_T^2 Y_{Tk \ominus i} Y_{Tk \ominus j}}.$$

Substituting $\frac{P_{T}^{i,k} P_{T}^{k,j} P_{T}^{j,k} P_{T}^{k,i}}{Y_T^2 Y_{Tk \ominus i} Y_{Tk \ominus j} Y_{Tk \oplus j}}$ for $Y_{T k \oplus i}$ into \eqref{E:Fk3} (using $Y_{Tik} = Y_{Tk \oplus i} Y_{Tk \ominus i}$), we obtain up to a monomial factor
$$ \frac{P_{T}^{i,k} P_{T}^{k,j} P_{T}^{j,k} P_{T}^{k,i}}{Y_T^2 Y_{Tk \ominus i} Y_{Tk \cup j}} Y_{Tk \ominus i} Y_{T} + P_{T}^{i,k} P_{T}^{k,i}.$$ Killing the common factor $P_{T}^{i,k} P_{T}^{k,i}$ with $\hF_i|_{Y_{T\oplus k} = 0}$ and bringing to a common denominator, we observe that this is up to a monomial factor the exchange polynomial $\hF'_k$ in the cluster $t_{\s'}$ given by \eqref{E:hFi}.  The required result then follows from Proposition \ref{prop:rigid}.

%Consider the case when the graph has all possible edges, in other words $P_T^{i,j}$, $P_T^{j,i}$, $P_T^{i,k}$, $P_T^{k,i}$, $P_T^{k,j}$, $P_T^{j,k}$ are all non-zero. Then plugging in $Y_{T \oplus k} = 0$ into 
%$$\hF_i = \frac{Y_{T \cup k \cup i \cup j} Y_{T \cup k} + P_{T \cup k}^{i,j} P_{T \cup k}^{j,i}}{Y_{(T \cup k) \ominus i} Y_{(T \cup k) \ominus j}}$$
%kills the $Y_{T \cup k \cup i \cup j} Y_{T \cup k} = Y_{T \cup k \cup i \cup j} Y_{T \oplus k} Y_{T \ominus k}$ term. As for the other term, we use again Lemma \ref{lem:ikj} to write 
%$$P_{T \cup k}^{i,j} = \frac{P_{T}^{i,k} P_{T}^{k,j} + P_{T}^{i,j} Y_{T \cup k}}{Y_T} \;\;\; \text{and} \;\;\; P_{T \cup k}^{j,i} = \frac{P_{T}^{j,k} P_{T}^{k,i} + P_{T}^{j,i} Y_{T \cup k}}{Y_T},$$ where plugging $Y_{T \oplus k} = 0$ again kills one of the terms. Thus the whole result of plugging is 
%$$\frac{P_{T}^{i,k} P_{T}^{k,j} P_{T}^{j,k} P_{T}^{k,i}}{Y_T^2 Y_{(T \cup k) \ominus i} Y_{(T \cup k) \ominus j}}.$$

\section{Examples}
\label{sec:examples}
\subsection{Paths}

Consider  the path $P_n$ on $[n]$ that consists of edges $i \longrightarrow i+1$ for $i \in [n-1]$ and $i \longrightarrow i-1$ for $i \in [2,n]$.  The LP algebra $\A_{P_n}$ is the {\it path LP algebra}.  Let $\A'_{P_n}$ be the rank $n-1$ LP algebra defined in Section \ref{ss:graphassoc}.  The seeds of $\A'_{P_n}$ are exactly the seeds of $\A_{P_n}$ that contain $Y_{[n]}$, and with $Y_{[n]}$ removed.

The strongly connected components $P_n$ are exactly the intervals $\{[a,b] \mid 1 \leq a \leq b \leq n\}$.  There is a bijection between maximal nested families of $P_n$ with support $[n]$, and triangulations of an $(n+1)$-gon. The bijection is obtained by mapping a set $[a,b]$ into the diagonal which cuts vertices $a$ through $b$ away from the rest of the vertices. An example is given in Figure \ref{fig:lin2}.  The complex of partial triangulations of the $n+1$-gon give a simplicial complex dual to the associahedron.

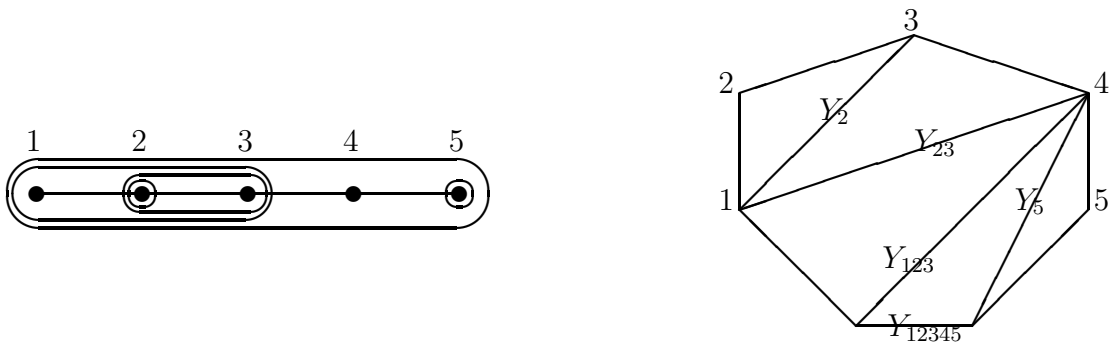
\begin{figure}[ht]
    \begin{center}
\setlength{\unitlength}{2pt}
\begin{picture}(240,60)(-10,0)
\thicklines
\put(5,30){\circle*{3}}
\put(3,38){$1$}
\put(5,30){\line(1,0){80}}
\put(25,30){\circle*{3}}
\put(23,38){$2$}
\put(45,30){\circle*{3}}
\put(43,38){$3$}
\put(65,30){\circle*{3}}
\put(63,38){$4$}
\put(85,30){\circle*{3}}
\put(83,38){$5$}
\put(25,30){\oval(5,5)}
\put(35,30){\oval(27,7)}
\put(25,30){\oval(49,10)}
\put(85,30){\oval(5,5)}
\put(45,30){\oval(91,13)}

\put(160,5){\line(1,0){22}}
\put(160,5){\line(-1,1){22}}
\put(182,5){\line(1,1){22}}
\put(171,60){\line(3,-1){33}}
\put(171,60){\line(-3,-1){33}}
\put(138,27){\line(0,1){22}}
\put(204,27){\line(0,1){22}}
\put(138,27){\line(1,1){33}}
\put(204,49){\line(-1,-2){22}}
\put(204,49){\line(-1,-1){44}}
\put(204,49){\line(-3,-1){66}}
\put(134,27){$1$}
\put(134,49){$2$}
\put(169,61){$3$}
\put(205,49){$4$}
\put(205,27){$5$}
\put(153,44){$Y_2$}
\put(171,38){$Y_{23}$}
\put(165,16){$Y_{123}$}
\put(190,27){$Y_{5}$}
\put(166,3){$Y_{12345}$}

\end{picture}
\qquad
    \end{center} 
    \caption{The nested collection $\{2,23,123,5,12345\}$ on the left corresponds to the triangulation on the right.
}
    \label{fig:lin2}
\end{figure}

\begin{theorem}\label{thm:path}
Let $t_\s$ be a seed of $\A'_{P_n}$.  Suppose $S_i = I \cup \{i\} \cup J \in \s$, where $I, J \in \s \cup \{\emptyset\}$.  Then the exchange relation for $Y_{S_i}$ has the form
$$Y_{IiJ} Y_{JjK} = Y_{J} Y_{IiJjK} + Y_{I} Y_{K}.$$
where $K \in \s \cup \{\emptyset\}$ is such that $j = i^+$ and $S_j = I \cup \{i\} \cup J \cup \{j\} \cup K$.
\end{theorem}
\begin{proof}
The claim follows immediately from the formula for exchange Laurent polynomials in  \eqref{E:hFi}.
\end{proof}

The {\it {Ptolemy cluster algebra}} of type $A_{n-1}$ is often modeled by by identifying clusters with triangulations of a $[n+1]$-gon and cluster variables with diagonals.  See for example \cite{Sch} for complete details.  In \cite[Section 4]{LP} we explained how certain cluster algebras could be considered as LP algebras.

\begin{corollary}
The LP algebra $\A'_{P_n}$ can be identified with the Ptolemy cluster algebra.  The identification between seeds and between cluster variables is given by  the bijection between maximal nested collections and triangulations.  
\end{corollary}

\subsection{Cycles}

Consider the cycle $C_n$ on $[n]$ that consists of edges $i \longrightarrow i \pm 1 \mod n$.  The LP algebra $\A_{C_n}$ is the {\it cycle LP algebra}.  Let $\A'_{C_n}$ be the rank $n-1$ LP algebra defined in Section \ref{ss:graphassoc}.  The seeds of $\A'_{C_n}$ are exactly the seeds of $\A_{C_n}$ that contain $Y_{[n]}$, and with $Y_{[n]}$ removed.

The strongly connected components $P_n$ are exactly the cyclic intervals $\{[a,b]\}$.  There is a bijection between maximal nested families of $C_n$ with support $[n]$ and centrally symmetric triangulations of a $2n$-gon. The bijection is obtained by mapping a set ${[a,b]}$ (indices taken modulo $n$) into the diagonal which cuts vertices $a$ through $b$ away from the rest of vertices. An example is given in Figure \ref{fig:lin3} where the labeling of the vertices of the $2n$-gon is shown.  The complex of partial centrally symmetric triangulations of the $2n$-gon give a simplicial complex dual to the cyclohedron.

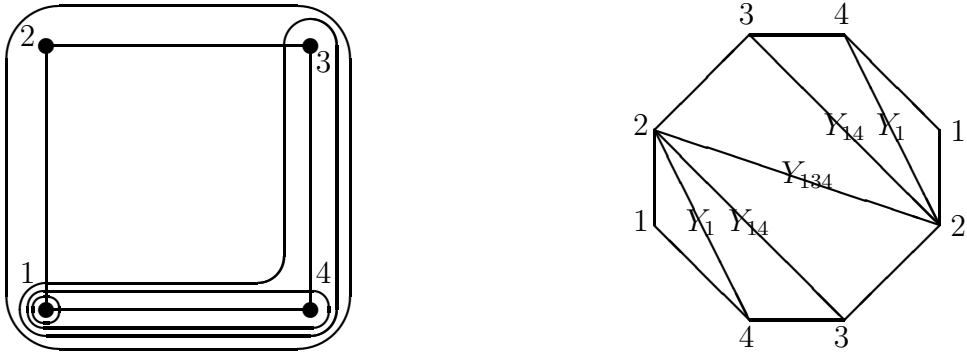
\begin{figure}[ht]
    \begin{center}
\setlength{\unitlength}{2pt}
\begin{picture}(240,60)(10,0)
\thicklines
\put(45,5){\circle*{3}}
\put(40,10){$1$}
\put(45,5){\line(1,0){50}}
\put(45,5){\line(0,1){50}}
\put(95,55){\line(-1,0){50}}
\put(95,55){\line(0,-1){50}}
\put(45,55){\circle*{3}}
\put(40,55){$2$}
\put(95,5){\circle*{3}}
\put(96,50){$3$}
\put(95,55){\circle*{3}}
\put(96,10){$4$}
\put(45,5){\oval(5,5)}
\put(70,5){\oval(57,7)}
\put(70,5){\oval(60,10)[b]}
\put(70,5){\oval(60,10)[tl]}
\put(70,15){\oval(40,10)[br]}
\put(85,30){\oval(10,40)[br]}
\put(95,30){\oval(10,60)[r]}
\put(95,30){\oval(10,60)[t]}
\put(70,30){\oval(65,65)}

\put(178,3){\line(1,0){18}}
\put(178,3){\line(-1,1){18}}
\put(160,39){\line(0,-1){18}}
\put(160,39){\line(1,-2){18}}
\put(160,39){\line(1,-1){36}}
\put(160,39){\line(3,-1){54}}
\put(160,39){\line(1,1){18}}
\put(178,57){\line(1,-1){36}}
\put(178,57){\line(1,0){18}}
\put(196,57){\line(1,-1){18}}
\put(196,57){\line(1,-2){18}}
\put(214,21){\line(0,1){18}}
\put(214,21){\line(-1,-1){18}}

\put(156,20){$1$}
\put(156,38){$2$}
\put(176,59){$3$}
\put(194,59){$4$}
\put(216,37){$1$}
\put(216,19){$2$}
\put(194,-2){$3$}
\put(176,-2){$4$}
\put(184,29){$Y_{134}$}
\put(192,38){$Y_{14}$}
\put(174,20){$Y_{14}$}
\put(202,38){$Y_{1}$}
\put(166,20){$Y_{1}$}

\end{picture}
\qquad
    \end{center} 
    \caption{An example of the bijection between maximal nested sets on a cycle and centrally symmetric triangulations.
}
    \label{fig:lin3}
\end{figure}

%Let 
%$j \in S$, $j \neq i$ be the other vertex such that $Jj$ is connected, and let $K \in \s$, $K \cap J = \emptyset$ be the largest element of $\s \cup \emptyset$ such that $jK$ is connected.

\begin{theorem}
Let $t_\s$ be a seed of $\A'_{C_n}$, and let $S_i = IiJ \neq [n]$, where $I,J \in \s \cup \{\emptyset\}$ are chosen so that if $j = i^+$ then $S_j = IiJjK$ for some $K \in \s \cup \{\emptyset\}$ (the concatenated elements and subsets are disjoint and ordered around the cycle).  The exchange relation for $Y_{S_i}$ is described as follows:
\begin{itemize}
 \item if $S_{j^+} \neq [n]$, we have 
$$Y_{IiJ} Y_{JjK} = Y_{J} Y_{IiJjK} + Y_{I} Y_{K};$$
 \item if $S_{j^+} = [n]$, so that $IiJj = [n]$, we have 
$$Y_{IiJ} Y_{JjI} = Y_I Y_{J} Y_{[n]} + (Y_{I}+Y_{J})^2.$$
\end{itemize}
\end{theorem}

\begin{proof}
 Direct application of \eqref{E:hFi}.
\end{proof}

The LP algebra $\A'_{C_n}$ cannot be identified with a type $B$ cluster algebra, even though it has the same cluster complex (dual to the cyclohedron).

\begin{example}
 The seed corresponding to the maximal nested collection in Figure \ref{fig:lin3} has the following form:
$$\{(Y_{134}, (1+Y_{14})^2+Y_{1234}Y_{14}), (Y_{14}, Y_{134}+Y_1), (Y_1, 1+Y_{14})\}.$$
\end{example}

In particular, we see that some of the exchange polynomials are not binomial, and thus this cannot be identified with a type $B$ cluster algebra.

\subsection{Complete graphs}

Consider the complete graph $K_n$ on $[n]$ that consists of all edges $i \longrightarrow j$ for $i \not = j$. The clusters are in bijection with activation sequences $\vec{S} = (s_1, \ldots, s_k)$, each sequence yielding a different cluster. 

For $1 \leq j \leq k$ define $\YY_{S_{s_j}}$ by 
$$\YY_{S_{s_j}} = \prod_{\ell =1}^{j-1} Y_{S_{s_{\ell}}}^{2^{j-\ell-1}}$$
and define $\Yy_{S_{s_j}}$ recursively via
$$\Yy_{S_{s_1}} = 1+ Y_{s_1} \;\;\; \text{and} \;\;\; \Yy_{S_{s_j}} = \prod_{1 \leq \ell < j} \Yy_{S_{s_{\ell}}} + Y_{S_{s_j}} \YY_{S_{s_{j-1}}}.$$
For example, $\Yy_{S_{s_4}}=$
$$(1+ Y_{s_1})(1+ Y_{s_1} +Y_{s_1s_2})((1+ Y_{s_1})(1+ Y_{s_1} +Y_{s_1s_2})+Y_{s_1} Y_{s_1s_2s_3})+Y_{s_1}^2 Y_{s_1s_2} Y_{s_1s_2s_3s_4}.$$

\begin{theorem}
The exchange polynomials of the seed $t_{\vec{S}}$ of $\A_{K_n}$ are given by
\
\begin{itemize}
 \item For $j < k$ we have $$F_{S_{s_j}} = \prod_{1 \leq \ell < j} \Yy_{S_{s_{\ell}}}^2 + Y_{S_{s_{j+1}}} \YY_{S_{s_{j}}} \;\;\; \text{and} \;\;\; \hF_{S_{s_j}} = \frac{F_{S_{s_j}}}{\YY_{S_{s_{j-1}}}^2}.$$
 \item For $j=k$ we have
$$F_{S_{s_k}} = A_{s_k} Y_{S_{s_{k-1}}} \YY_{S_{s_{k-1}}} + \sum_{1 \leq j<k} (A_{s_j} Y_{S_{s_{j-1}}} \YY_{S_{s_{j-1}}} \cdot \prod_{1 \leq \ell < k, \ell \not = j} \Yy_{S_{s_{\ell}}}) + \left(\sum_{j \not \in S} X_j \right) \prod_{1 \leq \ell < k} \Yy_{S_{s_{\ell}}}$$
$$\text{and} \;\;\; \hF_{S_{s_k}} = \frac{F_{S_{s_k}}}{\YY_{S_{s_{k-1}}}}.$$
 \item For $m \not \in S$ we have
$$F_{{m}} = A_{m} Y_{S_{s_{k}}} \YY_{S_{s_{k}}} + \sum_{1 \leq j \leq k} (A_{s_j} Y_{S_{s_{j-1}}} \YY_{S_{s_{j-1}}} \cdot \prod_{1 \leq \ell \leq k, \ell \not = j} \Yy_{S_{s_{\ell}}}) + \left(\sum_{j \not \in S, j \not = m} X_j \right) \prod_{1 \leq \ell \leq k} \Yy_{S_{s_{\ell}}}$$
$$\text{and} \;\;\; \hF_{m} = \frac{F_{m}}{\YY_{S_{s_{k}}}}.$$
\end{itemize}
\end{theorem}

The proof is meticulous but straightforward, and it is omitted. 

\begin{example}
 Take $n=4$ and $\vec{S} = (2,4,1)$. Then the variables are $(Y_{124}, Y_2, X_3, Y_{24})$, and the seed $t_{\vec{S}}$ looks like 
$$\{(Y_{124}, A_1 Y_{24} Y_2 + A_4 Y_2 (1+Y_2) + A_2(1+Y_2 +Y_{24}) + X_3 (1+Y_2)(1+Y_2 +Y_{24})),$$ $$(Y_2, 1+Y_{24}),$$
$$(X_3, A_3 Y_{124} Y_{24} Y_2^2 + A_1 Y_{24} Y_2(1+Y_2)(1+Y_2 +Y_{24})+ A_4 Y_2 (1+Y_2)((1+Y_2)(1+Y_2 +Y_{24})+Y_{124}Y_{2}) +$$  $$A_2(1+Y_2 +Y_{24})((1+Y_2)(1+Y_2 +Y_{24})+Y_{124}Y_{2})),$$ $$(Y_{24}, (1+Y_2)^2 + Y_{124}Y_2)\}.$$
\end{example}

\section{Conjectures}
\label{sec:conjectures}
The following are some questions and conjectures about graph LP algebras and more generally linear LP algebras. 

\begin{conjecture}\label{con:pos}
Let $\A_\Gamma$ be a graph LP algebra and $t_\s$ a seed.  Then
\begin{enumerate}
\item
All the exchange polynomials $F_i$ of $t_\s$ are positive polynomials in $\L(t_\s)$.
\item
Any cluster variable of $\A_\Gamma$ has positive coefficients when written as a Laurent polynomial in $\L(t_\s)$.
\end{enumerate}
\end{conjecture}
The Laurentness of Conjecture \ref{con:pos}(2) follows from Theorem \ref{thm:LP}.
Conjecture \ref{con:pos}(2) implies Conjecture \ref{con:pos}(1).  For the case $t_\s = t_\emptyset$ is the initial seed, Conjecture \ref{con:pos}(2) follows from the definition of the $Y_I$.  The calculations in Section \ref{sec:examples} also support Conjecture \ref{con:pos}.

\begin{problem}
 Find a combinatorial interpretation for the coefficients of the Laurent polynomials in the Conjecture \ref{con:pos}.
\end{problem}

A monomial in the cluster variables of $\A_\Gamma$ is a {\it {cluster monomial}} if there is some seed $t_\s$ containing all the variables occurring in that monomial. 

\begin{conjecture}\label{con:cluster}
Let $\A_\Gamma$ be a graph LP algebra. 
\begin{enumerate}
\item
Cluster monomials form a basis for $\A_\Gamma$ over $R$.
\item
 Any monomial in the cluster variables of $\A_\Gamma$ can be written as a nonnegative linear combination of cluster monomials.
\end{enumerate}
\end{conjecture}

\begin{problem}
 Find a combinatorial interpretation for the coefficients of cluster monomials in  Conjecture \ref{con:cluster}.
\end{problem}

%We believe it would be interesting to find a categorification of linear LP algebras, which could answer some of the questions above.
In this paper we have focused on exchange polynomials of the form $F_i = A_i + \sum_{i \to j} X_j$.  Allowing the $A_i$ to be specialized, or the $X_j$ to have coefficients gives an arbitrary linear LP algebra.

\begin{conjecture}\label{con:finite}
Any linear LP algebra is of finite type.
\end{conjecture}

Conjectures \ref{con:pos} and \ref{con:cluster} should also have analogues for linear LP algebras whose coefficients satisfy a positivity constraint.

Our results have established that the cluster complex of a graph LP algebra $\A_\Gamma$ is isomorphic to the extended nested complex $\hN(\Gamma)$.  We conjecture that $\hN(\Gamma)$ is dual to a the face lattice of a polytope.  More generally,

\begin{conjecture}\label{con:polytope}
Let $\A$ be any linear LP algebra.  Then there exists a convex simple polytope $P(\A)$ of dimension equal to the rank of $\A$ such that
\begin{enumerate}
\item the vertices of $P(\A)$ are in bijection with clusters of $\A$
\item the facets of $P(\A)$ are in bijection with cluster variables of $\A$
\end{enumerate}
and under these bijections we have that for each face $F$ of $P(\A)$ the collection of facets containing $F$ are exactly the set of cluster variables contained in the intersection of the clusters corresponding to the vertices of $F$.
\end{conjecture}
In particular, Conjecture \ref{con:polytope} would imply that the exchange graph of any linear LP algebra is the 1-skeleton of a polytope. 
The graph associahedra of Carr and Devadoss \cite{CD} and more generally the nestohedra of Postnikov \cite{Pos} show that the LP algebras $\A'_\Gamma$ satisfy Conjecture \ref{con:polytope}.

\section{Proof of Theorem \ref{thm:comm}} \label{sec:proof4}
For polynomials $P,Q$ and a variable $x$, define $\den(P,x,Q)$ to be
$$
\den(P,x,Q) = \min_i(i+\max(s : Q^s | p_i)) 
$$
where $P(x) = \sum_{i=0}^r p_i x^i$.  By definition we have
$$
\hF_i = \prod_{j \neq i} x_j^{-\den(F_i,x_j,F_j)} F_i.$$  
Note that $\den(P,x,Q)$ makes sense even when $P$ is a Laurent polynomial in some variables, as long as $Q$ is not divisible by any of these variables.

\begin{lemma}
Suppose $Q$ is irreducible.  Then
$$
\den(P,x,Q) + \den(P',x,Q) = \den(PP',x,Q).
$$
\end{lemma}
\begin{proof}
The inequality $\leq$ is clear.  Let $d = \den(P,x,Q) + \den(P',x,Q)$.  Let $i_0$ be the minimal $i$ such that the minimum $\min_i(i+\max(s : Q^s | p_i))$ is achieved, and similarly for $j_0$ and $\min_j(j+\max(s : Q^s | p'_j))$.  Set $k = i_0+j_0$.  Then the coefficient of $x^{k}$ in $PP'$ is $p_0p'_k + \cdots + p_{i_0}p'_{j_0} + \cdots + p'_kp_0$.  By the minimality assumption, $Q^{d-k+1}$ divides all the terms except for $p_{i_0}p'_{j_0}$, which is only divisible by $Q^{d-k}$.  Thus $\den(PP',x,Q) \leq d$.
\end{proof}

Start with a seed $t = \{(a,P),(b,Q),(x_1,F_1),\ldots,(x_k,F_k)\}$.  We make the assumption that $P$ may depend on $b$, and that $Q$ does not depend on $a$.  Furthermore, we assume that $P/Q$ is not a unit, so in particular that $\den(Q,a,P) = 0$.  We need to construct the successive mutations of $t$ at $a$ and $b$.   In the case that we only have two cluster variables, the claim we want was established in \cite[Proposition 6.3]{LP}.

\subsection{}
We mutate the seed $t$ at $b$ to get $\{(a,R),(b',Q),\ldots,(x_i,F'_i),\ldots\}$.  Our first calculation is that $\hP = \hR$ after we substitute $b'=\hQ/b$.  Strictly speaking, we mean that we can choose the mutation at $b$ so that $\hP = \hR$, since $R$ is only defined up to units.  We will abuse notation in this way throughout.

First we may write
$$
Q = \hQ \times \prod_i x_i^{\den(Q,x_i,F_i)} = \hQ \times D.
$$
 Write 
$$
P(b) = p_0+p_1b + \cdots + p_rb^r.
$$
Then since $Q$ does not depend on $a$, we have
$$
R(b') = \frac{p_0(Db')^r + p_1 (Db')^{r-1}Q + \cdots + p_rQ^r}{Q^sE} = P(\hQ/b')\frac{(Db')^r}{Q^sE}
$$
for some $0 \leq s \leq r$ and $E$ some monomial in the $x_i$.  Also we have
$$
\hP = \dfrac{P}{b^s \prod_i x_i^{\den(P,x_i,F_i)}}
$$
and from the calculation in \cite[Proposition 6.3]{LP}, we already know that $\den(R,b,Q) = r-s$.
For each $i$ define $d_i = \den(Q,x_i,F_i)$.  Note that the monomial $E$ above only involve $x_i$ such that $d_i > 0$.  Let $R' = \dfrac{P(\hQ/b')(Db')^r}{Q^s}$.

\begin{lemma}\label{L:dRP}
We have $\den(R',x_i,F'_i) = d_i(r-s) + \den(P,x_i,F_i)$.
\end{lemma}
\begin{proof}
Case 1: Suppose $F_i$ does not depend on $b$.  Then $F'_i = F_i$.  Let $R'' = P(\hQ/b')(Db')^r$.  Then 
$$
\den(R'',x_i,F_i) = \min_j(d_i(r-j)+ \den(Q^jp_j,x_i,F_i)) =\min_j(d_ir+ \den(p_j,x_i,F_i)) = d_ir + \den(P,x_i,F_i).
$$
Thus $\den(R',x_i,F'_i)=\den(R''/Q^s,x_i,F_i) = d_i(r-s) + \den(P,x_i,F_i)$. 
%The $x_i^{d(r-s)}$ cancels with the part of $(Db')^r$, the remaining part matching with the numerator of $$\frac{P(\hQ/Db')}{({Q}/{Db'})^s}.$$ 
%Doing it for every $x_i$ we see that $\hP = \hR$.

Case 2: Suppose $F_i$ depends on $b$.  Then $d_i = 0$ and $F'_i$ depends on $b'$.  Also $x_i$ does not occur in $D$ or $E$.  We shall prove that $\den(P,x_i,F_i) = \den(R,x_i,F'_i)$ which suffices since $\den(R,x_i,F'_i) = \den(R',x_i,F'_i)$. 
By symmetry, it is enough to show that $\den(R,x_i,F'_i) \geq \den(P,x_i,F_i)$.  Let $P = a_0+a_1x_i+\cdots+a_kx_i^k$.  If $F_i^s$ divides $a_j$ then we have $a_j = F_i^s\,G$ and
$$
a_j(\hQ/b') = F_i(\hQ/b')^s G(\hQ/b') = (F_i(\hQ(0)/b')+ x_i H)^s G(\hQ/b')
$$
where $\hQ(0)$ is the evaluation of $\hQ$ at $x_i = 0$, and $H$ is a polynomial in $x_i$.  Since $F'_i$ is a divisor of $F_i(\hQ(0)/b')$ up to a monomial in the other variables, we 
see that $\den(a_j(\hQ/b'),x_i,F'_i) \geq s$.  It follows that $\den(R,x_i,F'_i) \geq \den(P,x_i,F_i)$.
\end{proof}

Using the Lemma \ref{L:dRP}, we calculate
\begin{align*}
\hR  &= \dfrac{R'}{b'^{r-s}\prod_i x_i^{\den(R',x_i,F'_i)}} \\
&=\dfrac{P (Db')^r}{b'^{r-s}(Db'b)^s \prod_i x_i^{\den(R',x_i,F'_i)}} \\
&=\dfrac{P  \prod_i x_i^{d_i(r-s)}}{b^s  \prod_i x_i^{\den(R',x_i,F'_i)}} \\
&=\dfrac{P}{b^s \prod_i x_i^{\den(P,x_i,F_i)}} = \hP
\end{align*}

Given the seed $t$, we can mutate at $a$, then $b$, then at the new variable $a'$ at the position of $a$, then at the new variable $b'$ at the position at $b$.   We can also mutate first at $b$, then $a$, then $b'$, then $a'$.  In the following, we call this ``mutating four times''.

\begin{lemma}\label{L:a}
Mutating four times we recover the cluster variable $a$.
\end{lemma}
\begin{proof}
We have shown that $\hP = \hR$.  Thus if we mutate $\{(a,R),(b',Q),\ldots\}$ at $a$ to get $\{(a',R),(b',S),\ldots\}$ the variable $a' = \hR/a = \hP/a$ is the same as what we get if we mutate $\{(a,P),(b,Q),\ldots\}$ at $a$.  But mutating one more time at the position of $b$ (on either end) will not change the first cluster variable.  This is true for all seeds $t'$ with the the exchange polynomial of $b$ does not depend on $a$.  The claim follows.
\end{proof}

\subsection{}
Our next task is to show that mutating four times preserves $b$.   We first mutate $\{(a,P),(b,Q),\ldots\}$ at $a$ to get a seed $\{(a',P),(b,S),\ldots\}$.  But $Q$ does not depend on $a$, so we may pick $S = Q$.  We now show that $\hQ = \hS$, which would follow from $\den(Q,x_i,F_i) = \den(Q,x_i,F'_i)$ where now $F'_i$ is the exchange polynomial obtained by mutation at $a$.  Suppose first that $F_i$ depends on $a$, implying that $F'_i$ also depends on $a$.  Then since $Q$ does not depend on $a$, we have $\den(Q,x_i,F_i) = 0 = \den(Q,x_i,F'_i)$.  Otherwise $F_i$ does not depend on $a$, and so $F_i = F'_i$, and we have $\den(Q,x_i,F_i) = \den(Q,x_i,F'_i)$.  The same argument as in Lemma \ref{L:a} gives

\begin{lemma}
Mutating four times we recover the cluster variable $b$.
\end{lemma}

\subsection{}
We now show that the exchange polynomials $P$ and $Q$ are recovered after four mutations.  For $Q$ that is clear since mutation at $a$ or $b$ does not change $Q$.  Now consider the sequence of 3 mutations at $b$,$a$,$b'$, giving $\{(a,R),(b',Q),\ldots\}$, then $\{(a',R),(b',Q),\ldots\}$, then $\{(a',T),(b,Q),\ldots\}$.  We need to show that $T = P$.  But we know that one more mutation at $a$ gives us $\{(a,T),(b,Q),\ldots\}$, using the fact that $a$ is recovered after four mutations.  But we thus have $a'= \hT/a = \hP/a$ where $\hT$ can be calculated in the seed $\{(a',T),(b,Q),\ldots\}$, while $\hP$ is calculated in the seed $\{(a,P),(b,Q),\ldots\}$.  Here we have used that $a'$ is also recovered after four mutations.  Thus $\hT = \hP$.  Both are Laurent polynomials in $b$ and $x_i$, so we must have $T = P$ as polynomials in $b$ and $x_i$.

\subsection{}
Finally we need to show that $(x_i,F_i)$ is recovered after four mutations.  This is clear for $x_i$.  Note that it is enough to recover $F_i$ up to a monomial factor in the variables $a,b,x_j$, so we can just work with Laurent polynomials throughout.  Consider the sequence of seeds $t_1=\{(a,P),(b,Q),\ldots\}$, $t_2 = \mu_b(t_1) = \{(a,R),(b',Q),\ldots\}$, then $t_3 = \mu_a(t_2) = \{(a',R),(b',Q),\ldots\}$, then $t_4 = \mu_{b'}(t_3)= \{(a',P),(b,Q),\ldots\}$, then $t_5 = \mu_{a'}(t_4) = \{(a,P),(b,Q),\ldots\}$.  Fix $i$, and let the corresponding $F_i$ in each seed be denoted $$
Z_1(a,b),Z_2(a,b'),Z_3(a',b'),Z_4(a',b),Z_5(a,b).$$  We need to show that $Z_1 = Z_5$.  We shall first assume that $Z_1$ depends on $a$ and $b$.  We have (up to a monomial factor in $a,b,x_j$)
$$
Z_2(a,b') = Z_1(a,\hQ(0)/b')/A
$$
where $A$ is some product of factors of $\hQ(0)$ not depending on $a,b',x_i$.  Here $\hQ(0)$ is the evaluation of $\hQ(x_i)$ at $x_i = 0$.  Similarly,
$$
Z_3(a',b') = Z_2(\hR(0)/a',b')/B(b') =\frac{Z_1(\hR(0)/a',\hQ(0)/b')}{AB(b')}
$$
where since $Z_2$ is defined up to a monomial factor, so is $Z_3$ up to a monomial factor in $a',b',x_j$, and in addition $B(b')$ is defined up to a power of $\hR(0)$ since $Z_2$ was defined up to a power of $a$.  Continuing,
$$
Z_4(a',b) = \frac{Z_1(\hR(0)/a',b)}{AB(\hQ(0)/b)C}
$$
where we have used that $\hQ$ is well-defined regardless of the seed it is calculated in, and we have the usual ambiguities.  In addition, the substitution $b = \hQ(0)/b'$ is also made in $\hR(0)$.  (Note that it is possible for $Z_3(a,b)$ to not depend on $b$, but the argument continues in essentially the same way.) Finally,
$$
Z_5(a,b) = \frac{Z_1(\frac{\hR(0)}{\hP(0)/a},b)}{AB(\hQ(0)/b)CD(b)}
$$
Now we already showed previously that $\hR=\hP$, so we have that $Z_5(a,b)$ and $Z_1(a,b)$ are equal up to a monomial factor in $a,b,x_j$, together with the many factors $A,B,C,D$ none of which depend on $a$, and which are only defined up to certain factors which also do not depend on $a$.  Since $Z_1(a,b)$ depends on both $a$ and $b$, we must have $Z_5(a,b) = Z_1(a,b)$.  

In the case that $Z_1(a,b)$ does not depend on $b$ but does depend on $a$ the argument is similar.  When $Z_1(a,b)$ does not depend on $a$ at all, the argument is straightforward.

This completes the proof of Theorem \ref{thm:comm}.

\end{document}